\documentclass[leqno,11pt,a4paper]{amsart}
\usepackage[top=28truemm,bottom=28truemm,left=38truemm,right=38truemm]{geometry}
\usepackage{fix-cm}
\usepackage[TU]{fontenc}
\makeatletter
\renewcommand\section{\@startsection{section}{1}{\z@}%
           {25\p@ \@plus 6\p@ \@minus 3\p@}%
           {10\p@ \@plus 6\p@ \@minus 3\p@}%
           {\fontsize{13pt}{0cm}\selectfont\bfseries\boldmath}}
           
\renewcommand\subsection{\@startsection{subsection}{2}{\z@}%
           {13\p@ \@plus 6\p@ \@minus 3\p@}%
           {6\p@ \@plus 6\p@ \@minus 3\p@}%
           {\fontsize{12pt}{0cm}\bfseries\boldmath}}
\renewcommand\subsubsection{\@startsection{subsubsection}{3}{\z@}%
           {12\p@ \@plus 6\p@ \@minus 3\p@}%
           {\p@}%
           {\normalfont\normalsize}}
\renewcommand{\paragraph}[1]{%
  \par
  \addvspace{\medskipamount}
  \noindent
  \textbf{#1\@addpunct{.}}\enspace\ignorespaces
}
\makeatother
\let\oldtocsection=\tocsection
\let\oldtocsubsection=\tocsubsection
\let\oldtocsubsubsection=\tocsubsubsection
\renewcommand{\tocsection}[2]{
\hspace{0em}\bfseries\oldtocsection{#1}{#2}}
\renewcommand{\tocsubsection}[2]{\hspace{1em}\small\oldtocsubsection{#1}{#2}}
\renewcommand{\tocsubsubsection}[2]{\hspace{2em}\small\oldtocsubsubsection{#1}{#2}}
\usepackage[x11names,dvipsnames]{xcolor}
\usepackage{url}
\usepackage[doi=false,isbn=false, sortcites=true,eprint=true, date=year, style=numeric,uniquename=init,giveninits=true,sorting=nty,url=true,maxnames=99,backref=false]{biblatex} 
\usepackage[hypertexnames=false,linktoc=page]{hyperref}
\hypersetup{
    colorlinks=true,
    citecolor=blue,
    linkcolor=blue,
    urlcolor=teal,
    breaklinks=true
}
\usepackage{amsfonts}
\usepackage[capitalize,noabbrev,nameinlink]{cleveref}

\usepackage{amsthm}
\usepackage{thmtools}

\numberwithin{equation}{section}

\theoremstyle{plain}
\newtheorem{theorem}{Theorem}[section]
\newtheorem{proposition}[theorem]{Proposition}
\newtheorem{lemma}[theorem]{Lemma}
\newtheorem{corollary}[theorem]{Corollary}

\theoremstyle{definition}
\newtheorem{definition}[theorem]{Definition}

\newtheorem{assumption}[theorem]{Assumption}

\theoremstyle{remark}
\newtheorem{remark}[theorem]{Remark}
\newtheorem{example}[theorem]{Example}

\usepackage[inline]{enumitem}
\setlist[description]{%
  topsep=10pt,               
  itemsep=5pt,               
  font={\bfseries\rmfamily}, 
}

\setlist[itemize]{
  left=0pt,
  itemsep=5pt,
}
\setlist[enumerate]{
    left=0pt,
    itemsep=5pt,
}

\newlist{myenum}{enumerate}{3}
\setlist[myenum,1]{label=\textbf{\arabic*.},
                   ref  =\textbf{\arabic*.}}
\setlist[myenum,2]{label=\textbf{(\alph*)},
                   ref  =\themyenumi\textbf{(\alph*)}}
\setlist[myenum,3]{label=\bfseries(\roman*),
                   ref  =\themyenumii\textbf{.(\roman*)}}

\usepackage{mathtools}
\usepackage{amssymb,amsmath}
\usepackage{mathrsfs}
\usepackage{appendix}

\crefname{theorem}{Theorem}{Theorems}
\crefname{definition}{Definition}{Definitions}
\crefname{problem}{Problem}{Problems}
\crefname{fact}{Fact}{Facts}
\crefname{proposition}{Proposition}{Propositions}
\crefname{lemma}{Lemma}{Lemmas}
\crefname{corolary}{Corolary}{Corolaries}
\crefname{conjecture}{Conjecture}{Conjectures}
\crefname{assumption}{Assumption}{Assumptions}
\crefname{enumi}{Assumption}{Assumptions}
\crefname{claim}{Claim}{Claims}

\crefname{remark}{Remark}{Remarks}
\crefname{example}{Example}{Examples}
\crefname{corollary}{Corollary}{Corollaries}
\crefname{subsubsection}{Subsubsection}{Subsubsections}
\crefname{subsection}{Subsection}{Subsections}
\crefname{section}{Section}{Sections}
\crefname{chapter}{Chapter}{Chapters}
\crefname{table}{Table}{Tables}
\crefname{figure}{Figure}{Figures}

\crefname{myenumi}{item}{items}
\crefname{myenumii}{item}{items}
\crefname{myenumiii}{item}{items}
\DeclareMathOperator*{\argmin}{\mathop{\rm arg~min}}
\DeclareMathOperator{\sgn}{\operatorname{sgn}}

\newcommand{\N}{\mathbb{N}}
\newcommand{\R}{\mathbb R}
\newcommand{\dd}{\mathrm{d}}

\newcommand{\lrbra}[1]{\left(#1\right)}

\newcommand{\eps}{\varepsilon}
\newcommand{\e}{\mathrm{e}}

\newcommand{\domain}[1]{\mathsf{D}(#1)}

\newcommand{\inner}[2]{\left\langle #1, #2 \right\rangle}

\newcommand{\X}{\mathcal{X}}
\newcommand{\dist}{\mathsf{d}}

\newcommand{\m}{\mathsf{m}}
\newcommand{\Leb}{\mathscr{L}}

\newcommand{\Lp}{L^p(X, \m)}

\newcommand{\vdual}{V^\ast}
\newcommand{\normslope}[1]{\norm{\partial_F^\circ #1}}

\newcommand{\norm}[1]{\left\lVert #1 \right\rVert}
\newcommand{\lpnorm}[1]{\lVert #1 \rVert_{L^p}}
\newcommand{\lqnorm}[1]{\left\lVert #1 \right\rVert_{L^q}}
\newcommand{\lonenorm}[1]{\lVert #1 \rVert_{L^1}}

\newcommand{\slope}[1]{\abs{\partial #1}}

\newcommand{\F}{\mathcal{F}}

\newcommand{\tconsume}[1]{\mathsf{t}_{#1}}
\newcommand{\tmin}{\tconsume{\mathrm{min}}}

\newcommand{\relaxedslope}[1]{\left|\mathrm{D} #1 \right|_\ast}

\newcommand{\ch}[1]{\mathsf{Ch}_{#1}}
\newcommand{\chp}{\ch{p}}

\newcommand{\sobolevnorm}[1]{\|#1\|_{H^{1,p}(X)}}
\newcommand{\sobolev}{H^{1,p}(X)}

\newcommand{\plap}{\Delta_p}

\newcommand{\sphere}{\mathrm{S}}
\newcommand{\ball}{\mathrm{B}}

\newcommand{\bregman}[2]{D_\Psi \bigl(#1, #2\bigr)}
\newcommand{\resolvent}[1]{J^\tau[#1]}
\newcommand{\xbar}[2]{\bar{x}^{#1}_{#2}}

\newcommand{\Psiast}{\Psi^\ast}

\newcommand{\bregmanast}[2]{D_{\Psiast}\bigl(#1, #2\bigr)}
\newcommand{\gradpsi}{\nabla\Psi}

\newcommand{\xtau}{x^\tau}
\newcommand{\xtaun}{\xtau_n}
\newcommand{\xtaunminusone}{\xtau_{n-1}}
\newcommand{\xbartau}[1]{\bar{x}^\tau_{#1}}

\newcommand{\abs}[1]{\left\lvert #1 \right\rvert}
\newcommand{\dt}{\frac{\dd}{\dd t}}

\usepackage{comment}

\DeclareRobustCommand{\SkipTocEntry}[5]{}

\begin{document}
\title{Mirror flows and $p$-Laplacian eigenvalue problems on metric measure spaces}

\author[S.~Shimoyama]{Sho Shimoyama} 
\subjclass[2020]{
Primary 47J10;
Secondary 47J35, 49J35, 49J52, 35B40
}
\renewcommand{\thefootnote}{\fnsymbol{footnote}}

\keywords{
$p$-Laplacian eigenvalue problem;
metric measure space;
mirror flow;
large time behavior
}

\address{
Addresses of Sho Shimoyama
\endgraf
Graduate School of Mathematical Sciences,
The University of Tokyo,
Meguro-ku Komaba 153-8914,
Japan
}
\email{\href{mailto:sho-shimoyama@g.ecc.u-tokyo.ac.jp}{sho-shimoyama@g.ecc.u-tokyo.ac.jp}}

\begin{abstract}
We study mirror flows as Banach-space counterparts of Hilbertian gradient flows and show that they play an essential role in $p$-Laplacian eigenvalue problems on metric measure spaces.
Under a Rellich--Kondrachov type compactness assumption and for $p\ge2$, we establish a Ljusternik--Schnirelman type existence theorem without assuming $C^1$-regularity of the associated energy. 
More precisely, we prove that every element $\lambda$ of the Krasnoselskii spectrum is an eigenvalue of the $p$-Laplacian $\Delta_p$; that is, there exists a nontrivial solution $f$ to $\Delta_p f=-\lambda |f|^{p-2}f$.
We also investigate the large time behavior of the corresponding mirror flow and prove its convergence to an eigenfunction when the eigenvalue is simple and isolated.
\end{abstract}

\maketitle

\tableofcontents

\section{Introduction}
The well-posedness of gradient flows for convex functionals on Hilbert spaces, including uniqueness and continuous dependence on initial data, is a fundamental property underlying many applications in optimization, analysis and geometry.
In non-Hilbertian Banach spaces, however, these properties remain open, which limits the applicability of gradient flows; for related topic see~\cite[p.~4]{AGS} and~\cite{Ohta2023ContractionProperty}.
This paper proposes viewing mirror flows as a natural Banach-space analogue of Hilbertian gradient flows; the results for $p$-Laplacian eigenvalue problems on metric measure spaces provide evidence in favor of this viewpoint.

Roughly speaking, mirror flows are dual counterparts of standard gradient flows.
Let $(V, \norm{\cdot})$ be a Banach space, $\Phi \colon V \to (-\infty, +\infty]$ be a proper lower semicontinuous functional, and let $\Psi \colon V \to \R$ be a $C^1$-convex functional.
We assume that $\Phi$ is $\lambda$-convex for $\Psi$ and some $\lambda \in \R$, which means that $\Phi - \lambda\Psi$ is convex.
A $\Psi$-mirror flow $x_t \colon [0, +\infty) \to V$ for $\Phi$ is a curve satisfying
\[
    \dt\nabla\Psi(x_t) \in - \partial_F \Phi(x_t),\ \Leb\text{-a.e. } t > 0,
\]
where $\nabla\Psi$ is the Fr\'echet differential of $\Psi$, $\partial_F \Phi$ is the Fr\'echet subdifferential of $\Phi$ and $\Leb$ is the $1$-dimensional Lebesgue measure.
For comparison with gradient flows, recall the definition of $\Psi$-gradient flow: a $\Psi$-gradient flow $u_t \colon [0, +\infty) \to V$ for $\Phi$ is a curve with
\[
    \nabla\Psi\left(\dt u_t \right) \in - \partial_F \Phi(u_t),\ \Leb\text{-a.e. } t > 0.
\]
The term involving $\nabla\Psi$ induces a strong nonlinearity, which makes the analysis of $\Psi$-gradient flows delicate.
Indeed, even in Hilbert spaces, for $\Psi(\cdot)=\frac{1}{p}\norm{\cdot}^p$ with $p\ne2$, the uniqueness of the flow was pointed out in~\cite[p.~208]{RossiSegattiStefanelli2011GlobalAttractors} as an open problem, and was recently settled affirmatively in~\cite{Shimoyama2025} via a parameter transformation, a nonstandard technique in the analysis of gradient flows.
The important difference between $\Psi$-mirror flows and $\Psi$-gradient flows is the order of the time derivative $\dt$ and the nonlinear map $\nabla\Psi$.
Thanks to this difference, $\Psi$-mirror flows are compatible with the chain rule.
For instance, if $\Psi(\cdot) = \frac{1}{2}\norm{\cdot}^2$, $\Phi$ is convex, i.e. $0$-convex for $\Psi$, and $\Phi$ has a minimizer $x^\ast = 0 \in V$, then for $\Leb$-a.e. $t >0 $ we have
\[
    \dt \frac{1}{2}\norm{x_t - x^\ast}_V^2
    = \dt \frac{1}{2}\norm{\nabla\Psi(x_t - x^\ast)}_{V^\ast}^2
    = \inner{\dt \nabla\Psi(x_t)}{x_t - x^\ast} \le \Phi(x^\ast) - \Phi(x_t) \le 0.
\]
Thus, the distance between $x_t$ and a minimizer $x^\ast$ is nonincreasing even if $V$ is not a Hilbert space.
In applied mathematics, particularly in optimization and control, $\Psi$-mirror flows have been widely studied and the importance of the above compatibility property was already recognized by Nemirovsky and Yudi in the 1980s~\cite{NemirovskyYudin1983ProblemComplexity}.
This study shows that $\Psi$-mirror flows are also useful for $p$-Laplacian eigenvalue problems on metric measure spaces.

Before treating general metric measure spaces, let us first recall the results in Euclidean spaces. 
Let $\Omega \subset \R^n$ be a bounded domain with smooth boundary, endowed with the Euclidean distance and the Lebesgue measure.
For $p \in (1, +\infty)$, denoted by $q$ the H\"older conjugate of $p$, the $p$-Laplace operator $\Delta_p \colon L^p(\Omega) \to L^q(\Omega)$ is defined by $\Delta_p f \coloneqq \partial \mathcal{E}_p(f)$ for
\[
 \mathcal{E}_p(f) \coloneqq \begin{cases}
     \frac{1}{p} \int_\Omega |\nabla f|^p \dd x & \text{if } f \in W^{1,p}(\Omega)\\
     +\infty & \text{otherwise}.
 \end{cases}
\]
As shown in~\cite{LE20061057}, a nondecreasing sequence of eigenvalues $\{\lambda_{k, p}\}_{k \ge 1}$ of $\Delta_p$ is given by the min-max values 
\begin{equation}\label{eq:minmax_in_intro}
    \lambda_{k, p} \coloneqq \inf_{A \in \F_k} \sup_{A} p\mathcal{E}_p.
\end{equation}
Here $\F_k$ consists of all closed symmetric subsets of $\sphere \coloneqq \{f \in L^p(\Omega) \mid \|f\|_{L^p} = 1\}$ having a topological dimension at least $k$ in the sense of Krasnoselskii.
This min--max construction is an instance of the Ljusternik--Schnirelmann theory, a standard variational framework for producing nonlinear eigenvalues.
In this framework, a key role is played by a \textit{deformation map}.
The existence of a deformation map is usually based on the $C^1$-regularity of the associated functional: see for instance~\cite{Zeidler1985,struwe_variational}.

The past three decades have seen progress in analysis and geometry on metric measure spaces which have become an active area of research.
In this article, by a metric measure space $(X, \dist, \m)$ we mean a metric space equipped with a finite Borel measure $\m$, i.e., $\m(X) < +\infty$.
The functional $\mathcal{E}_p$ is generalized in this setting by $p$-Cheeger energy $\chp \colon \Lp \to [0,+\infty]$.
However, since $\chp$ is not $C^1$ on its domain in general, we cannot use the same method as in the Euclidean setting to construct the deformation map.
In~\cite{AmbrosioHondaPortegies2018Continuity}, Ambrosio--Honda--Portegies construct a deformation map using gradient flows for the case $p=2$ and show that min-max values $\{\lambda_{2, k}\}_{k \ge 1}$ for $\ch{2}$ are eigenvalues of $2$-Laplacian under the assumption that the domain $H^{1,2}(X)$ of $\ch{2}$ is compactly embedded in $L^2(X, \m)$.
The well-posedness property of gradient flows in Hilbert spaces is important in their proof; however, this property are not known in $\Lp$ for $p \neq 2$.
The authors state that it is open whether the min-max values $\{\lambda_{k, p}\}_{k\ge 1}$ are eigenvalues of the $p$-Laplacian.

In this study, we extend their result to the case $p \ge 2$.
More precisely, we prove:
\begin{theorem}\label{thm:min_max_eigenvalue_introduction}
    Denoted by $\sphere$ the unit sphere in $\Lp$, namely, $\sphere \coloneqq \{ f \in \Lp \mid \lpnorm{f} = 1\}$.
    Let $p \ge 2$ and $\sobolev$ be the domain of $p$-Cheeger energy $\chp$. Assume that $\sobolev$ is compactly embedded in $\Lp$.
    Then for each $k \in \N$ the min-max value
    \[
        \lambda_{k, p} \coloneqq \inf_{A \in \F_k} \sup_A p\chp
    \]
    is an eigenvalue of the $p$-Laplacian if it is finite, where $\F_k$ is the set of all closed symmetric subsets $A \subset \sphere$ with $\gamma(A) \ge k$ and $\gamma$ is the Krasnoselskii genus.
    Moreover, a multiplicity result holds: for the precise statement, see \cref{thm:min_max_are_eigenvalues}.
\end{theorem}
The key part of the proof is the construction of a deformation map without assuming $C^1$-regularity of the associated energy.
This is established using $\Psi$-mirror flows for $\Psi(\cdot) \coloneqq \frac{1}{p}\lpnorm{\cdot}^p$ and a functional $\Phi$ on $\Lp$ which, for an appropriate choice of $M > 0$, is defined by 
\[
    \Phi(f) \coloneqq \begin{cases}
        \chp(f) - (M+1)\Psi(f) & \text{ if } \lpnorm{f} \le1 \text{ and } \chp(f) \le \frac{M}{p}\\
        +\infty & \text{otherwise}
    \end{cases}.
\]
In order to construct the deformation map, we prove the well-posedness of $\Psi$-mirror flows for this functional $\Phi$ in \cref{cor:unique_existence_mirror_flow}.
We remark that the validity of the analogous result for $\Psi$-gradient flows remains unclear and difficult to verify.

As another application, we study the large time behavior of $\Psi$-mirror flows for the above $\Phi$ and $\Psi$.
In~\cite{HYND20174873}, Hynd and Lindgren show that $\Psi$-gradient flows converge to a minimizer of the Rayleigh quotient for several functionals on $L^p(\Omega)$,  including $\mathcal{E}_p$, when the minimizer is simple and isolated.
Inspired by their work, we prove the following result, where we define
\[
    \widetilde{\lambda}_{1, p} \coloneqq \inf \left\{p\chp(f) \mid f \in \sphere \text{ with } \int_X J_p(f) \dd\m = 0\right\}.
\]
\begin{theorem}\label{thm:approximation_eigenvalue_introduction}
    Let $p \ge 2$, $f_0 \in \sphere$ satisfy $\chp(f_0) < +\infty$ and $\int_X J_p(f_0) \dd \m = 0$, and let $f_t \colon [0,+\infty) \to \Lp$ be a $\Psi$-mirror flow for $\Phi$ starting from $f_0$.
    Denoted by $\omega(f_t)$ the $\omega$-limit set of the curve $f_t$
    \[
        \omega(f_t) \coloneqq \{f_\ast \in \Lp \mid \exists t_n \to +\infty \text{ such that } \lpnorm{f_{t_n} - f_\ast} \to 0\}.
    \]
    Assume $\sobolev$ is compactly embedded in $\Lp$.
    Then the limit
    \[
        \lambda_\ast \coloneqq \lim_{t \to +\infty} p \chp(f_t)
    \]
    exists, is finite and satisfies $\widetilde{\lambda}_{1,p} \le \lambda_\ast$.
    In addition, $\lambda_\ast$ is an eigenvalue of the $p$-Laplacian and any element $f_\ast \in \omega(f_t)$ is an eigenvector for $\lambda_\ast$.
    
    Moreover, $\widetilde{\lambda}_{1,p}$ is an eigenvalue of $p$-Laplacian, and if $\widetilde{\lambda}_{1,p}$ is simple and isolated, and $p\chp(f_0)$ is sufficiently close to $\widetilde{\lambda}_{1,p}$, then there exists an eigenfunction $f_\ast \in \sphere$ for $\lambda_\ast$ such that
    \[
        \lpnorm{f_t - f_\ast} \to 0.
    \]
\end{theorem}
\begin{remark}\label{rem:about_simplicity_and_isolation}
    The above result concerns the Neumann boundary condition.
    However, the Dirichlet case can be treated in the same way.
    Specifically, the simplicity and isolation of the first Dirichlet eigenvalue are known in smooth spaces: for example, on any bounded domain in Euclidean space~\cite{Anane1987Simplicite,Lindqvist1990Equation,Lindqvist2008NonlinearEigenvalue}.
\end{remark}

We also establish an existence result for $\Psi$-mirror flows on a slightly more general setting covering the above $\Lp$ cases.
Before stating the result, we fix the setting.
Let $V$ be a Banach space such that its dual space $V^\ast$ has the Radon--Nikod\'ym property, and let $\Psi \colon V \to [0, +\infty)$ be a $C^1$-convex functional such that its convex conjugate $\Psi^\ast$ is also $C^1$ on $V^\ast$.
Let $\Phi \colon V \to (-\infty, +\infty]$ be a proper, lower semicontinuous and $\lambda$-convex functional for $\Psi$ and some $\lambda \in \R$, i.e., $\Phi - \lambda\Psi$ is convex.
We also suppose that \cref{ass:bregman_lower_inequality,ass:psi_norm_bounds,ass:bounded_from_below,ass:bounded_sublevel_compact} hold; in \cref{ass:bregman_lower_inequality} we define a map
\[
    \bregman{x}{y} \coloneqq \Psi(x) - \Psi(y) - \inner{\nabla\Psi(y)}{x-y} \text{ for each } x, y \in V,
\]
called the \textit{Bregman divergence}: for details, see~\cref{sec:bregman_divergence}.
\begin{theorem}\label{thm:existence_introduction}
    Let $\domain{\Phi}$ be the domain of $\Phi$.
    Under the above setting, for each $x_0 \in \domain{\Phi}$, there exists a $\Psi$-mirror flow $x_t$ for $\Phi$ starting from $x_0$.
\end{theorem}
Note that the existence of $\Psi$-mirror flows has been classically studied in the context of doubly nonlinear evolution equations: see for instance~\cite{GRANGE197277,AKAGI200632}.
More recently, a related formulation has been studied in terms of solutions to \textit{evolution variational inequalities} for general cost functions in~\cite{AUBINFRANKOWSKI2026111469}; we show the equivalence between $\Psi$-mirror flows and this notion in~\cref{thm:mirror_flow_equivalent_evi_solution}.
Our existence result does not require the continuity of $\Phi$ and the boundedness of $\partial \Phi$ on bounded sets, the subdifferential growth assumption for $\Phi$, or the cross-convexity and concavity assumptions imposed in~\cite{GRANGE197277,AKAGI200632,AUBINFRANKOWSKI2026111469}.
The proof is based on De Giorgi's minimizing movement scheme using the Bregman divergence $D_\Psi$ instead of the norm.
This idea is already implicit in~\cite{GRANGE197277}, although the assumptions there are different from those imposed in this study.
Notably, our aim is not to extend the existing theory in full generality, but to establish an existence result for $\Psi$-mirror flows under assumptions that are readily checkable in the metric measure space setting, considered above for $p$-Laplacian eigenvalue problems.

In addition to the existence result, we prove the following energy dissipation inequality for $\Psi$-mirror flows, which plays an important role in constructing the deformation map.
\begin{theorem}\label{thm:energy_dissipation_introduction}
    Let $x_t \colon [0, +\infty) \to V$ be a $\Psi$-mirror flow for $\Phi$.
    Assume that \cref{ass:bregman_lower_inequality} holds.
    Then we have
    \[
    \begin{aligned}
        - \dt \Phi(x_t)
        &\ge g(2\norm{\gradpsi(x_t)}) \norm{\dt \gradpsi(x_t)}^2 \text{ for } \Leb\text{-a.e. } t > 0,
    \end{aligned}
    \]
    where $g$ is the continuous function in \cref{ass:bregman_lower_inequality}.
\end{theorem}
To the best of our knowledge, the energy dissipation of $\Psi$-mirror flows was not previously known to hold in infinite-dimensional spaces.
Although analogous results are known under the symmetry of the Bregman divergence $D_\Psi$, in our setting it is generally not symmetric: see~\cite{AUBINFRANKOWSKI2026111469}.

In what follows, we first study $\Psi$-mirror flows on general Banach spaces in \cref{sec:properties,sec:existence}, and then apply these results to $p$-Laplacian eigenvalue problems in \cref{sec:p_laplacian}.

\section{Preliminaries}\label{sec:preliminaries}
\subsection{Absolutely continuous curves}
Let $(\X, \dist)$ be a complete metric space.
Denoted by $\Leb$ the one-dimensional Lebesgue measure.
A continuous map $x_t \colon [0, +\infty) \to \X$ is said to be a \textit{curve}.

\begin{definition}[locally absolutely continuous curve]
A curve $x_t \colon [0,+\infty) \to \X$ is said to be \textit{locally absolutely continuous} if there exists $A \in L^1_{\mathsf{loc}}([0,+\infty))$ such that
\[
    \dist(x_t, x_s) \le \int_s^t A(r) \dd r \text{ for each } 0 \le s \le t.
\]
\end{definition}

The proof of the following result can be found, for instance, in \cite{AGS}.
\begin{theorem}
    Let $x_t \colon [0,+\infty) \to X$ be a locally absolutely continuous curve.
    Then we have
    \begin{itemize}
        \item the metric speed $\abs{\dot{x}_t}$ of $x_t$, defined by 
        \[
            \abs{\dot{x}_t} \coloneqq \lim_{s \to t}\frac{\dist(x_t, x_s)}{|t - s|},
        \]
        exists and is finite for $\Leb$-a.e. $t >0$.
        \item the map $t \in [0, +\infty) \mapsto |\dot{x}_t|$ belongs to $L^1_{\mathsf{loc}}([0, +\infty))$.
        \item it holds that
        \[
            \dist(x_t, x_s) \le \int_s^t |\dot{x}_r| \dd r \text{ for each } 0 \le s \le t.
        \]
    \end{itemize}
\end{theorem}

\begin{remark}
    Let $\X$ be a reflexive Banach space (or more generally, a Banach space with the Radon--Nikod\'ym property) and $x_t \colon [0, +\infty) \to \X$ be a curve.
    Then the following are equivalent:
    \begin{enumerate}
        \item $x_t$ is a locally absolutely continuous curve.
        \item $x_t$ is differentiable at $\Leb$-a.e.~points $t > 0$, its derivative $\dot{x}_t$ satisfies that $t \mapsto \dot{x}_t$ belongs to $L^1_{\mathsf{loc}}([0,+\infty); \X)$ and it holds that
        \[
            x_t - x_s = \int_s^t \dot{x}_r \dd r \text{ for each } 0 \le s \le t.
        \]
    \end{enumerate}
    Moreover, in this case, we have
    \[
        \norm{\dot{x}_t} = |\dot{x}_t| \text{ for } \Leb\text{-a.e. } t > 0.
    \]
\end{remark}

\subsection{Bregman divergence}\label{sec:bregman_divergence}
Let $V$ be a Banach space with dual space $V^\ast$, and $\Psi \colon V \to \R$ be a convex $C^1$-functional such that its convex conjugate $\Psiast$ is also $C^1$ on $V^\ast$.
Denoted by $\nabla\Psi$ (resp. $\nabla \Psiast$) the Fr\'echet derivative of $\Psi$ (resp. $\Psiast$).

The following function originates from Bregman;~\cite{BREGMAN1967200}.
\begin{definition}[$\Psi$-Bregman divergence]
    The $\Psi$-Bregman divergence $D_\Psi \colon V \times V \to [0, +\infty)$ is defined to be that
    \begin{equation}\label{eq:bregman_div}
        \bregman{x}{y} \coloneqq \Psi(x)-\Psi(y) - \inner{\nabla \Psi (y)}{x - y} \text{ for each } x, y \in V.
    \end{equation}
\end{definition}
Note that the function $D_\Psi$ is not symmetric in general.
However, it satisfies the following properties, which are analogous to those of distance functions.
\begin{proposition}
    The following hold:
    \begin{itemize}
        \item (nondegeneracy) for any $x, y \in V$, if $x=y$, then $\bregman{x}{y} = 0$. Moreover, if $\Psi$ is strictly convex, the converse implication also holds.
        \item (three-point identity) for any $x, y, z \in V$ we have
            \begin{equation}\label{eq:threepoint-identity}
                \bregman{x}{y} = \bregman{x}{z} + \bregman{z}{y} + \inner{\nabla \Psi (z)- \nabla \Psi(y)}{x-z}.
            \end{equation}
        \item (dual symmetry) for any $x, y \in V$ we have
            \begin{equation}\label{eq:bregman_duality}
                \bregman{x}{y} = \bregmanast{\nabla \Psi (y)}{\nabla \Psi (x)}.
            \end{equation}
    \end{itemize}
\end{proposition}
\begin{proof}
    We show each statement in turn.

    \noindent
    \underline{\textbf{nondegeneracy}}: It is clear that $x=y$ implies $\bregman{x}{y} = 0$.
    We show the converse direction.
    By the Fenchel--Young inequality, we have
    \[
        \bregman{x}{y} = \Psi(x) + \Psiast(\nabla \Psi(y)) - \inner{\nabla \Psi(y)}{x}.
    \]
    Thus, $\bregman{x}{y} = 0$ implies $\nabla \Psi(y) = \nabla\Psi(x)$.
    This gives $y = x$. 

    \noindent
    \underline{\textbf{three-point identity}}: This directly follows from the definition.

    \noindent
    \underline{\textbf{dual symmetry}}: Substituting the equalities $\Psi(x) = - \Psiast(\nabla\Psi(x)) + \inner{\nabla\Psi(x)}{x}$ and $\Psi(y) = - \Psiast(\nabla\Psi(y)) + \inner{\nabla\Psi(y)}{y}$ into the definition of $D_\Psi$ gives the conclusion.
\end{proof}

\subsection{\texorpdfstring{$\lambda$}{lambda}-convex functions}
Let $(V, \norm{\cdot})$ be a Banach space and $\Phi \colon V \to (-\infty, +\infty]$ be a lower semicontinuous and proper functional, i.e., $\domain{\Phi} \coloneqq \{ x \in V \mid \Phi(x) < +\infty\} \neq \emptyset$.
Let $\Psi \colon V \to \R$ be a convex $C^1$-functional whose convex conjugate $\Psiast$ is also $C^1$ on $V^\ast$.

We define a notion of convexity for $\Phi$ with respect to $\Psi$ as follows:
\begin{definition}[$\lambda$-convex for $\Psi$]\label{def:lambda_convex}
    The functional $\Phi$ is said to be \textit{$\lambda$-convex} for $\Psi$ and some $\lambda \in \R$ if the map $x \in V \mapsto (\Phi - \lambda \Psi)(x)$ is convex.
\end{definition}

\begin{remark}
    It is unclear whether this notion should be regarded as a generalization or a special case of \textit{$(\lambda, p)$-convexity} of functions on metric spaces: see for instance \cite{RossiSegattiStefanelli2011GlobalAttractors}.
    As an example, we consider $L^p$-space as $V$.
    In this case, thanks to Corollary~2 of \cite{Xu1991Inequalities}, if $p \ge 2$ and $\Phi$ is $\lambda$-convex for $\Psi(x) \coloneqq \frac{1}{p}\lpnorm{x}^p$ and some $\lambda \ge 0$, then $\Phi$ is $(\lambda c_p, p)$-convex in the sense of~\cite{RossiSegattiStefanelli2011GlobalAttractors}, where $c_p$ is a constant depending only on $p$.
    However, it is not clear whether the converse holds.
\end{remark}

We denote the Fr\'echet subdifferential of $\Phi$ at $x \in V$ by
\[
    \partial_F \Phi(x) \coloneqq \left\{ \xi \in \vdual \mid \liminf_{y \to x} \frac{\Phi(y) - \Phi(x) -\inner{\xi}{y - x}}{\norm{y - x}} \ge 0\right\}.
\]
\begin{lemma}\label{lem:global_subgradient_formula}
    Let $x \in V$ and $\xi \in \vdual$.
    Let $\Phi$ be $\lambda$-convex for $\Psi$ and some $\lambda \in \R$.
    Then $\xi \in \partial_F\Phi(x)$ if and only if
    \[
        \inner{\xi}{y-x} + \lambda \bregman{y}{x} \le \Phi(y) - \Phi(x) \text{ for each } y \in \domain{\Phi}.
    \]
\end{lemma}
\begin{proof}
    \underline{$\Leftarrow$}: This direction is clear because $\Psi$ is $C^1$.

    \noindent \underline{$\Rightarrow$}: Since $F \coloneqq \Phi - \lambda \Psi$ is convex and $\partial F(x) = \partial_F \Phi(x) - \lambda \nabla\Psi(x)$, there exists a subgradient $\varphi \in \partial F(x)$ such that $\xi = \varphi + \lambda \nabla\Psi(x)$.
    Therefore, for any $y \in \domain{\Phi}$ we have 
    \[
    \inner{\xi}{y-x} = \inner{\varphi}{y - x} + \lambda \inner{\nabla\Psi(x)}{y - x}
    \le F(y) - F(x) + \lambda \inner{\nabla\Psi(x)}{y - x}.
    \]
    This completes the proof.
\end{proof}

The remaining part of this section is devoted into the study of some slopes of $\Phi$.
For each $x \in V$, we define the minimal-norm slope $\normslope{\Phi}$ by
\[
    \normslope{\Phi}(x) \coloneqq \begin{cases}
        \inf \{ \norm{\xi} \mid \xi \in \partial_F\Phi(x)\} & \text{if } \partial_F\Phi(x) \neq \emptyset\\
        +\infty & \text{otherwise}
    \end{cases},
\]
and the local slope $\slope{\Phi}$ by
\begin{equation}\label{eq:local_slope}
    \slope{\Phi}(x) \coloneqq \begin{cases}
        \limsup_{y \to x} \frac{(\Phi(x) - \Phi(y))^+}{\norm{y - x}} & \text{if } x \in \domain{\Phi}\\
        +\infty & \text{otherwise}
    \end{cases}.
\end{equation}

The local slope $\slope{\Phi}$ has the following global formula.
An analogous result holds for $(\lambda, p)$-convex functionals on metric spaces;~\cite[Proposition~2.7]{RossiSegattiStefanelli2011GlobalAttractors}.
\begin{proposition}\label{prop:lower_semicontinuity_of_slope}
    Let $\Phi$ be $\lambda$-convex for $\Psi$ and some $\lambda \in \R$.
    Then we have
    \begin{equation}\label{eq:global_formula}
        \slope{\Phi}(x) = \sup_{y \neq x} \lrbra{\frac{\Phi(x) - \Phi(y)}{\norm{y-x}} + \lambda\frac{\bregman{y}{x}}{\norm{y-x}}}^+ \text{ for each } x \in V.
    \end{equation}
    In particular, the map $x \in V \mapsto \slope{\Phi}(x)$ is lower semicontinuous.
\end{proposition}
\begin{proof}
    We show both inequalities in turn.

    \noindent \textbf{\underline{LHS $\le$ RHS}}: Since $\Psi$ is $C^1$, we have
    \[
        \begin{aligned}
            \slope{\Phi}(x) 
            &= \limsup_{y \to x}\lrbra{\frac{\Phi(x) - \Phi(y)}{\norm{y-x}} + \lambda\frac{\Psi(y) - \Psi(x) - \inner{\gradpsi(x)}{y - x}}{\norm{y-x}}}^+\\
            &\le \sup_{y \neq x} \lrbra{\frac{\Phi(x) - \Phi(y)}{\norm{y-x}} + \lambda\frac{\bregman{y}{x}}{\norm{y-x}}}^+.
        \end{aligned}
    \]

    \noindent \textbf{\underline{LHS $\ge$ RHS}}: We can assume that $\slope{\Phi}(x) < +\infty$.
    Let $y \neq x$ and $x_t \coloneqq (1-t) x + ty$ for $t \in [0,1]$.
    By the convexity of $\Phi - \lambda \Psi$, we have for each $y \in \domain{\Phi}$ with $y \neq x$
    \[
    \begin{aligned}    
        \frac{\Phi(x) - \Phi(y)}{\norm{y-x}} + \lambda \frac{\bregman{y}{x}}{\norm{y-x}}
        &\le \limsup_{t \downarrow 0}\lrbra{\frac{\Phi(x) - \Phi(x_t)}{\norm{x_t - x}} + \lambda \frac{\bregman{x_t}{x}}{\norm{x_t-x}}}\\
        &= \limsup_{t \downarrow 0} \frac{(\Phi(x) - \Phi(x_t))^+}{\norm{x - x_t}} \le \slope{\Phi}(x).
    \end{aligned}
    \]
    This completes the proof.
\end{proof}

The local slope $\slope{\Phi}$ coincides with the minimal norm slope $\normslope{\Phi}$.
\begin{proposition}\label{prop:slope_is_minimal_norm}
    Let $\Phi$ be $\lambda$-convex for $\Psi$ and some $\lambda$.
    Then we have
    \[
        \slope{\Phi}(x) = \normslope{\Phi}(x) \text{ for each } x \in V.
    \]
\end{proposition}
\begin{proof}
    We prove both inequalities in turn.
    
    \noindent \textbf{\underline{$\slope{\Phi} \le \normslope{\Phi}$}}: Without loss of generality, we can assume that $\partial_F\Phi(x) \neq \emptyset$.
    Let $\xi \in \partial_F\Phi(x)$.
    By the definition of the Fr\'echet subdifferential, for any $\eps > 0$, if $y \in X$ is sufficiently close to $x$, then we have
    \[
    \frac{\Phi(x) - \Phi(y)}{\norm{y-x}}\le \|\xi\|+\eps.
    \]
    This gives the conclusion.

    \noindent \textbf{\underline{$\slope{\Phi} \ge \normslope{\Phi}$}}: With no loss generality, we can assume that $\slope{\Phi}(x) < +\infty$.
    Let $F \coloneqq \Phi - \lambda \Psi$.
    By the global formula \eqref{eq:global_formula}, we have
    \[
        F(x) - F(y) + \lambda \inner{\gradpsi(x)}{x - y} \le \slope{\Phi}(x)\norm{y - x} \text{ for each } y \in V.
    \]
    Thus, by Hahn--Banach separation theorem, there exists $\xi \in \vdual$ such that
    \[
        F(x) - F(y) + \lambda \inner{\gradpsi(x)}{x - y}
        \le \inner{-\xi}{y - x}
        \le \slope{\Phi}(x)\norm{y - x} \text{ for each } y \in V
    \]
    The first inequality and \cref{lem:global_subgradient_formula} give that $\xi \in \partial_F\Phi(x)$.
    Moreover, the second inequality shows that $\norm{\xi} \le \slope{\Phi}(x)$.
    This completes the proof.
\end{proof}

\section{Properties of mirror flows}\label{sec:properties}
Let $(V, \norm{\cdot})$ be a Banach space such that its dual space $V^\ast$ has the Radon-Nikod\'ym property, and $\Phi \colon V \to (-\infty, +\infty]$ be a proper and lower semicontinuous functional.
Let $\Psi \colon V \to [0, +\infty)$ be a $C^1$-convex functional such that its convex conjugate $\Psiast$ is also $C^1$ on $\vdual$.

Throughout this section, we always assume that $\Phi$ is $\lambda$-convex for $\Psi$ and some $\lambda \in \R$.
We also impose the following assumption:
\begin{restatable}{assumption}{bregmanassumption}\label{ass:bregman_lower_inequality}
The following hold:
\begin{enumerate}
    \item there exist $\alpha > 1$ and $C_\alpha > 0$ such that
    \[
        \norm{y - x}^\alpha \le C_\alpha\bregman{y}{x} \text{ for each } y, x \in V;
    \]
    \item there exists a continuous function $g \colon (0,+\infty) \to (0, +\infty)$ such that
    \[
        g(0) \coloneqq \lim_{r \to 0} g(r) = +\infty
    \]
    and
    \[
        g(\norm{\varphi} + \norm{\psi}) \norm{\varphi - \psi}^2 \le \bregmanast{\varphi}{\psi} \text{ for each } \varphi, \psi \in V^\ast,
    \]
    where we set $\infty \cdot 0 \coloneqq 0$.
\end{enumerate}
\end{restatable}

\begin{remark}\label{rem:lp_bregman}
    Let $V$ be an $L^p$-space with $p \in [2, +\infty)$ and $\Psi(\cdot) \coloneqq \frac{1}{p}\norm{\cdot}_{L^p}^p$.
    Denoted by $q$ the H\"older conjugate of $p$.
    Then $\Psi$ and $\Psiast$ satisfy \cref{ass:bregman_lower_inequality} for $\alpha = p$, $C_\alpha = (p (p-1) 2^{p-1})^{-1}$ and $g(r) = \frac{q-1}{2}(r)^{q-2}$.
\end{remark}

\begin{remark}
    In \cref{sec:p_laplacian}, it is enough to consider $L^p$-spaces as $V$ and $\frac{1}{p}\lpnorm{\cdot}^p$ as $\Psi$.
    However, in this section we work in a more general setting.
\end{remark}

In this section, we study mirror flows, which are defined as follows.
\begin{definition}[$\Psi$-mirror flow]
    A map $x_t \colon [0, +\infty) \to V$ is said to be a \textit{$\Psi$-mirror flow} for $\Phi$ starting from $x_0 \in V$ if the map $t \in [0, +\infty) \mapsto \nabla\Psi(x_t)$ is locally absolutely continuous and satisfies the mirror inclusion
    \begin{equation}\label{eq:mirror_inclusion}
        \dt \nabla\Psi(x_t) \in - \partial_F \Phi(x_t) \text{ for } \Leb\text{-a.e. } t > 0.
    \end{equation}
\end{definition}

\begin{example}
    If $V$ is a Hilbert space with the Hilbert norm $\norm{\cdot}$ and $\Psi(\cdot) \coloneqq \frac{1}{2}\norm{\cdot}^2$, then $\Psi$-mirror flows are usual gradient flows.
    Further, in this case, the convexity notion of \cref{def:lambda_convex} is equivalent to the usual $\lambda$-convexity of $\Phi$.
\end{example}

In \cite{AUBINFRANKOWSKI2026111469}, a specific map $x_t \colon [0, +\infty) \to V$, closely related to mirror flows, is studied in a more general setting.
We quote their result in a form adapted to the present study although they treat a more general setting.
\begin{theorem}[{\cite[Theorem~2.4 and Remark~2.8]{AUBINFRANKOWSKI2026111469}}]\label{thm:equivalent_definition_of_EVI}
Let $\lambda \in \R$ and $x_t \colon [0, +\infty) \to V$ be a map.
Then, the following are equivalent:
\begin{enumerate}
    \item the map $t \in [0, +\infty) \mapsto x_t$ is continuous, $x_t \in \domain{\Phi}$ for any $t > 0$, and $x_t$ satisfies the \textit{Evolution Variational Inequality} (EVI) in the differential form:
    \begin{equation}\label{eq:evi_differential_form}
        \dt^+ \bregman{x}{x_t} + \lambda \bregman{x}{x_t} \le \Phi(x) - \Phi(x_t) \text{ for each } t \ge 0 \text{ and } x \in \domain{\Phi}.
    \end{equation}
    \item For each $x \in \domain{\Phi}$, the maps $t \mapsto \Phi(x_t)$ and $t \mapsto \bregman{x}{x_t}$ belong to $L^1_{\mathsf{loc}}((0,+\infty))$, $(s,t) \mapsto \bregman{x_s}{x_t}$ is Lebesgue measurable on $(0,+\infty)\times(0,+\infty)$, and $x_t$ satisfies the EVI in the integral form:
    \begin{equation}\label{eq:evi_integral_form}
        \bregman{x}{x_t} - \bregman{x}{x_s} + \lambda \int_s^t \bregman{x}{x_r} \dd r \le \int_s^t \Phi(x) - \Phi(x_r) \dd r \text{ for each } 0 \le s \le t.
    \end{equation}
    \item $x_t$ satisfies the EVI in the exponential integral form:
    \begin{equation}\label{eq:evi_exponential_integral_form}
        \e^{\lambda(t-s)}\bregman{x}{x_t} - \bregman{x}{x_s} \le E_\lambda(t-s) (\Phi(x) - \Phi(x_t)) \text{ for each } 0 \le s \le t \text{ and } x \in \domain{\Phi},
    \end{equation}
    where $E_\lambda(t) \coloneqq \int_0^t \e^{\lambda r} \dd r$.
\end{enumerate}
Furthermore, if $x_t$ satisfies one of the above conditions (i) - (iii), then the map $t \in [0, +\infty) \mapsto \Phi(x_t)$ is non-increasing.
\end{theorem}

\begin{definition}[{\cite[Definition~2.7]{AUBINFRANKOWSKI2026111469}}]\label{def:lambda_evi_solution}
    Let $\lambda \in \R$.
    A map $x_t \colon [0, +\infty) \to V$ satisfying \eqref{eq:evi_exponential_integral_form} in \cref{thm:equivalent_definition_of_EVI} is said to be a \textit{$\lambda$-EVI solution} for the pair $(D_\Psi, \Phi)$.
\end{definition}

\begin{example}
    Let $V$ be a Hilbert space and $\Psi(\cdot) \coloneqq \frac{1}{2}\norm{\cdot}^2$.
    Then, it holds that $\bregman{x}{y} = \frac{1}{2}\|x - y\|^2$.
    Therefore, the $\lambda$-EVI notion in the sense of \cref{thm:equivalent_definition_of_EVI} coincides with the usual $\lambda$-EVI notion with respect to the squared distance: see for instance~\cite{AGS,MURATORI2020108347}.
\end{example}
It is well known that, when $V$ is a Hilbert space and $\Psi(\cdot) = \frac{1}{2}\lpnorm{\cdot}^2$, the usual gradient flow is equivalent to a $\lambda$-EVI solution.
A similar result holds for $\Psi$-mirror flows.
\begin{theorem}\label{thm:mirror_flow_equivalent_evi_solution}
    Let $x_t \colon [0, +\infty) \to V$ be a map.
    Then the following are equivalent:
    \begin{enumerate}
        \item $x_t$ is a $\Psi$-mirror flow for $\Phi$;
        \item $x_t$ is a $\lambda$-EVI solution for $(D_\Psi, \Phi)$ and the map $t \in [0, +\infty) \mapsto \gradpsi(x_t)$ is locally absolutely continuous.
    \end{enumerate}
    In particular, if $x_t$ is a $\Psi$-mirror flow for $\Phi$, then the map $t \in [0, +\infty) \mapsto \Phi(x_t)$ is non-increasing.
\end{theorem}
\begin{proof}
\textbf{\underline{(i) $\Rightarrow$ (ii)}}: Let $x \in \domain{\Phi}$.
Note that the map $t \in [0,+\infty) \mapsto \bregman{x}{x_t}$ is locally absolutely continuous.
For any $t > 0$ such that $\dt \gradpsi(x_t) \in - \partial_F\Phi(x_t)$, the chain rule and the dual symmetry~\eqref{eq:bregman_duality} give
\begin{equation}\label{eq:derivative_of_bregman}
 \dt \bregman{x}{x_t} = \dt \bregmanast{\gradpsi(x_t)}{\gradpsi(x)} = \inner{-\dt\gradpsi(x_t)}{x - x_t}.
\end{equation}
Combining this with \cref{lem:global_subgradient_formula} shows
\[
\dt \bregman{x}{x_t} + \lambda \bregman{x}{x_t} \le \Phi(x) - \Phi(x_t) \text{ for } \Leb\text{-a.e. } t > 0.
\]
By integrating the above inequality, we have \eqref{eq:evi_integral_form}:
\[
    \bregman{x}{x_t} - \bregman{x}{x_s} + \lambda \int_s^t \bregman{x}{x_r} \dd r \le \int_s^t \Phi(x) - \Phi(x_r) \dd r \text{ for each } 0 \le s \le t.
\]
Since $\int_s^t \Phi(x_r) \dd r <+\infty$ and $t \mapsto \Phi(x_t)$ is lower semicontinuous, the map $t \mapsto \Phi(x_t) $ belongs to $ L^1_{\mathsf{loc}}((0,+\infty))$.
It is clear that the remaining parts of (ii) in \cref{thm:equivalent_definition_of_EVI} hold.
Thus, $x_t$ is a $\lambda$-EVI solution for $(D_\Psi, \Phi)$.

\noindent \textbf{\underline{(ii) $\Rightarrow$ (i)}}: This implication immediately follows from \eqref{eq:evi_differential_form}, \eqref{eq:derivative_of_bregman} and \cref{lem:global_subgradient_formula}.
This completes the proof.
\end{proof}

Thanks to \cref{thm:mirror_flow_equivalent_evi_solution}, we obtain the following energy dissipation inequality.
We now restate \cref{thm:energy_dissipation_introduction} in a more detailed form:
\begin{theorem}[Energy dissipation inequality]\label{thm:energy_dissipation}
    Let $x_t \colon [0, +\infty) \to V$ be a $\Psi$-mirror flow for $\Phi$.
    Assume that \cref{ass:bregman_lower_inequality} holds.
    Then we have
    \[
    \begin{aligned}
        - \dt \Phi(x_t)
        \ge \limsup_{h \downarrow 0} \frac{D_\Psi(x_{t-h}, x_t)}{h^2} 
        \ge g(2\norm{\gradpsi(x_t)}) \norm{\dt \gradpsi(x_t)}^2
        \ge g(2\norm{\gradpsi(x_t)}) \slope{\Phi}^2(x_t),
    \end{aligned}
    \]
    for $\Leb$-a.e. $t>0$, where $g$ is the continuous function in \cref{ass:bregman_lower_inequality}.
\end{theorem}
\begin{proof}
    We show each inequality in turn.
    Note that $\Phi(x_t)$ is differentiable excluding at most countable set because of \cref{thm:mirror_flow_equivalent_evi_solution}.
    Therefore, it suffices to check that both inequalities hold for any $t > 0$ at which $\dt \Phi(x_t)$ exists and the inclusion $\dt \gradpsi(x_t) \in - \partial_F \Phi(x_t)$ holds.
    Let $t > 0$ be any such point.

    \noindent \textbf{\underline{First inequality}}: By \cref{thm:mirror_flow_equivalent_evi_solution}, $x_t$ is a $\lambda$-EVI solution for $(D_\Psi, \Phi)$; in particular, $x_t$ satisfies the EVI in the exponential integral form \eqref{eq:evi_exponential_integral_form}.
    Inserting $x_s \coloneqq x_{t-h}$ and $x \coloneqq x_{t-h}$ gives
    \[
        \e^{\lambda h} \bregman{x_{t-h}}{x_t} \le E_\lambda(h)( \Phi(x_{t-h}) - \Phi(x_t)).
    \]
    Dividing both sides by $h^2$ and taking limsup w.r.t. $h \downarrow 0$ gives the first inequality.

    \noindent \textbf{\underline{Second and third inequalities}}: By \eqref{eq:bregman_duality} and \cref{ass:bregman_lower_inequality}, we have
    \[
        \bregman{x_{t-h}}{x_t}
        \ge g(\norm{\gradpsi(x_{t})} + \norm{\gradpsi(x_{t-h})}) \norm{\gradpsi(x_t) - \gradpsi(x_{t-h})}^2 \text{ for any } h > 0.
    \]
    Thus, dividing both sides by $h^2$ and taking limsup w.r.t.~$h\downarrow 0$, \cref{prop:slope_is_minimal_norm} and the continuity of $g$ give the second inequality.
    \cref{prop:slope_is_minimal_norm} completes the proof.
\end{proof}

\section{Existence of mirror flows by minimizing movements with Bregman divergence}\label{sec:existence}
Let $(V, \norm{\cdot})$ be a Banach space such that its dual space $V^\ast$ has the Radon-Nikod\'ym propety, and $\Psi \colon V \to [0, +\infty)$ be a convex $C^1$-functional such that its convex conjugate $\Psiast$ is also $C^1$ on $V^\ast$.
Let $\Phi \colon V \to (-\infty, +\infty]$ be a proper, lower semicontinuous and $\lambda$-convex functional for $\Psi$ and some $\lambda \in \R$.

In this section, we prove \cref{thm:existence_introduction}, which is restated in a more precise form as~\cref{thm:existence_MF}.
For the reader's convenience, we collect the assumptions used in this section. 
While all of them are imposed in the proof of the existence of $\Psi$-mirror flows, each auxiliary proposition and lemma below is stated only under the assumptions needed for its proof.

\bregmanassumption*

\begin{assumption}\label{ass:psi_norm_bounds}
    There exist $\beta > 0, C_\beta > 0$ and $m_\beta \ge 0$ such that
    \[
        \norm{\gradpsi(x)} \le C_\beta \norm{x}^\beta + m_\beta \text{ for each } x \in V.
    \]
\end{assumption}

\begin{assumption}\label{ass:bounded_from_below}
    It holds that
    \[
        m_0 \coloneqq \inf_{x \in V} \Phi(x) > -\infty
    \]
\end{assumption}

\begin{assumption}\label{ass:bounded_sublevel_compact}
    $(V, \Phi)$ has bounded sublevel compact sets, i.e., for each $r>0$ and $c \in \R$ the set $\overline{B_r} \cap \{\Phi \le c\}$ is compact, where $\overline{B}_r$ is the closed ball in $V$ with radius $r > 0$ and center $0$.
\end{assumption}

To show the existence of $\Psi$-mirror flows, we adopt De Giorgi's minimizing movement scheme~\cite{DeGiorgi1993NewProblems}.
As discussed in~\cite{AGS}, in the approach to usual gradient flows via minimizing movements, the variational scheme is formulated using the squared distance.
In the present setting, we replace the distance term by the $\Psi$-Bregman divergence.

For each $x, y \in V$ and $\tau > 0$, we define
\[
    \begin{aligned}
        \Phi_\Psi(y; x, \tau) & \coloneqq \Phi(y) + \frac1\tau \bregman{y}{x},\\
        \Phi^\tau_\Psi(x) & \coloneqq \inf_{y \in V} \Phi_\Psi(y; x, \tau),
    \end{aligned}
\]
and 
\[
    \resolvent{x} \coloneqq \argmin_{y\in V} \Phi_\Psi(y; x, \tau) \coloneqq \{y \in V \mid \Phi_\Psi(y; x, \tau) = \Phi^\tau_\Psi(x)\}.
\]
For each $x_0 \in V$ and $\tau > 0$, we define a sequence $\{x^\tau_n\}_{n \ge 1} \subset V$ to be that
\begin{equation}\label{eq:minimizing_sequence}
    x^\tau_n \in \resolvent{x^\tau_{n-1}},\ x^\tau_0 \coloneqq x_0,
\end{equation}
and then for this sequence $\{x^\tau_n\}_{n\ge1}$, we define the piecewise constant map $\xbar{\tau}{t} \colon [0, +\infty) \to V$ by
\begin{equation}\label{eq:interpolated_map}
    \xbar{\tau}{t} \coloneqq x^\tau_n \text{ if  } t \in ((n-1)\tau, n\tau],\ \xbar{\tau}{0} \coloneqq x^\tau_0.
\end{equation}
The existence of $\Psi$-mirror flows is obtained by passing to the limit in $\xbar{\tau}{t}$ as $\tau \to 0$.
\begin{remark}
    In general, $\resolvent{x}$ is not single-valued.
    However, none of the arguments in this section depends on the choice of an element in $J_\tau[x]$.
\end{remark}

\begin{remark}
    By a simple calculation, we can see that
    \begin{equation}\label{eq:discrete_differential_inclusion}
        0 \in \partial_F\Phi(x^\tau_n) + \frac{\nabla\Psi(x^\tau_n) - \nabla\Psi(x^\tau_{n-1})}{\tau} \text{ for each } n \ge 1.
    \end{equation}    
\end{remark}

\begin{remark}
    Under (i) of \cref{ass:bregman_lower_inequality}, \cref{ass:bounded_from_below,ass:bounded_sublevel_compact}, it holds that for any $x \in V$ and $\tau \in (0, +\infty)$
    \[
        \resolvent{x} \neq \emptyset.
    \]
\end{remark}

\subsection{A priori estimates}\label{sec:apriori_estimate}
Throughout this section, we assume the existence of a sequence $\{x^\tau_n\}_{n\ge1}$ given by \eqref{eq:minimizing_sequence}, hence that of $\xbar{\tau}{t}$ given by \eqref{eq:interpolated_map}, and focus on their properties.

First, we obtain the discrete version of EVI in the differential form: recall \eqref{eq:evi_differential_form}.
\begin{proposition}\label{prop:discrete_EVI_differential}
    If the sequence $\{x^\tau_n\}_{n\ge1}$ as in \eqref{eq:minimizing_sequence} exists, then it holds that
    \[
        \frac{\bregman{x}{\xtaun} - \bregman{x}{\xtaunminusone}}{\tau} + \lambda \bregman{x}{\xtaun} \le \Phi(x) - \Phi(\xtaun) \text{ for each } n \ge1, x \in \domain{\Phi}.
    \]
\end{proposition}
\begin{proof}
    The three-point identity \eqref{eq:threepoint-identity} and \cref{lem:global_subgradient_formula} give
    \[
        \begin{aligned}
        \bregman{x}{\xtaun} - \bregman{x}{\xtaunminusone}
        \le \tau \inner{\xi_n^\tau}{x - \xtaun}
        \le  \tau(\Phi(x) - \Phi(\xtaun)) - \lambda \tau \bregman{x}{\xtaun},
        \end{aligned}
    \]
    where $\xi_n^\tau \coloneqq -\frac{\gradpsi(\xtaun) - \gradpsi(\xtaunminusone)}{\tau}$.
    This completes the proof.
\end{proof}

We define
    \[
        \lambda^\tau \coloneqq \frac{\log(1 + \lambda\tau)}{\tau} \text{ for each } \tau < \frac{1}{\lambda^-},
    \]
and for each $t > 0$ define $n_t^\tau \ge 1$ to be that $t \in I^\tau_{n^\tau_t} \coloneqq ((n_t^\tau - 1)\tau, n_t^\tau \tau]$.
Note that the map $\tau \mapsto \lambda^\tau$ is non-increasing.

We also obtain the discrete version of EVI in the exponential integral form.
\begin{proposition}\label{prop:discrete_EVI_exponential}
    If the sequence $\{x^\tau_n\}_{n\ge1}$ as in \eqref{eq:minimizing_sequence} exists, then it holds that
    for each $0 \le m \le n$ and $x \in \domain{\Phi}$
    \[
        \e^{\lambda^\tau (n - m)\tau}\bregman{x}{\xtaun} - \bregman{x}{x^\tau_m} \le \int_{m\tau}^{n\tau} \e^{\lambda^\tau(n^\tau_r -1 -m)\tau}(\Phi(x) - \Phi(\xbar{\tau}{r})) \dd r.
    \]
\end{proposition}
\begin{proof}
    Iterating \cref{prop:discrete_EVI_differential}, we have
    \[
        \begin{aligned}
            \bregman{x}{\xtaun} 
            & \le \e^{-\lambda^\tau(n-m)\tau}\bregman{x}{x^\tau_m} + \sum_{i=m+1}^n \e^{-\lambda^\tau(n + 1 - i)\tau}\int_{I_i^\tau} (\Phi(x) - \Phi(\xbar{\tau}{r}))\dd r \\
            & = \e^{-\lambda^\tau(n-m)\tau}\bregman{x}{x^\tau_m} + \int_{m\tau}^{n\tau} \e^{-\lambda^\tau(n + 1 - n^\tau_r)\tau} (\Phi(x) - \Phi(\xbar{\tau}{r}))\dd r
        \end{aligned}
    \]
    This completes the proof.
\end{proof}

The next result immediately follows from that $\lambda^\tau \ge 0$ if $\lambda \ge 0$ and $\lambda^\tau < 0$ if $\lambda < 0$.
\begin{corollary}\label{cor:discrete_sequence_bound_from_minimizer}
    If $\{x^\tau_n\}_{n\ge1}$ as in \eqref{eq:minimizing_sequence} exists and \cref{ass:bounded_from_below} holds, then we have
    \[
        \bregman{x}{\xtaun} \le \e^{-\lambda^\tau n\tau}\bregman{x}{x^\tau_0} + n\tau \e^{(\lambda^\tau)^- n \tau}|\Phi(x) - m_0|,
    \]
    for each $ n \ge 1$ and $x \in \domain{\Phi}$, where $m_0$ is as in \cref{ass:bounded_from_below}.
\end{corollary}

As the last of this section, we study the family of maps $\{\xbar{\tau}{t}\}_{0 < \tau \ll 1}.$
We define $|D^\tau_t| \coloneqq \frac{\norm{\gradpsi(\xtaun) - \gradpsi(\xtaunminusone)}}{\tau}$ if $t \in I_n^\tau$ and $|D^\tau_t| \coloneqq 0$ if $t = 0$.

\begin{lemma}\label{lem:apriori_boounds}
    Assume that \cref{ass:bregman_lower_inequality,ass:psi_norm_bounds,ass:bounded_from_below} hold.
    Let $y_0 \in \domain{\Phi}$, $\tau_0 \in (0, +\infty)$ and $T > 0$.
    Then there exists a constant $C = C(\lambda, \alpha, \beta, m_0, T, \tau_0, y_0, x_0) > 0$ such that
    \[
        \sup_{t \in [0, T]}\Phi(\xbartau{t}) \le C,\ 
        \sup_{t \in [0, T]} \norm{\xbartau{t}} \le C \text{ and }
        \int_{0}^{T} |D^\tau_r|^2 \dd r \le C
    \]
    for any $\tau \le \tau_0$ and $\xbar{\tau}{t}$ given by \eqref{eq:interpolated_map}, where $x_0$ is the initial point of $\xbar{\tau}{t}$ and $m_0$ is as in \cref{ass:bounded_from_below}.
\end{lemma}
\begin{proof}
    Since $t \mapsto \Phi(\xbartau{t})$ is non-increasing, we have
    \[
        \sup_{t \in [0, T]} \Phi(\xbartau{t}) \le \Phi(x_0) < +\infty.
    \]
    Using (i) of \cref{ass:bregman_lower_inequality} and inserting $x \coloneqq y_0$ in \cref{cor:discrete_sequence_bound_from_minimizer}, for each $t \in [0, T]$ we have
    \[
    \norm{\xbartau{t}}^\alpha \lesssim_{\alpha} \norm{y_0}^\alpha + \bregman{y_0}{\xbartau{t}}
    \le \norm{y_0}^\alpha + \e^{(\lambda^{\tau_0})^-(T + \tau_0)} \lrbra{\bregman{y_0}{x_0} + (T+\tau_0)|\Phi(y_0) - m_0|},
    \]
    where $\lesssim_\alpha$ means the inequality holds up to a constant depending only on $\alpha$ as in \cref{ass:bregman_lower_inequality}.
    Combining this inequality with \cref{ass:psi_norm_bounds} gives
    \[
        \sup_{t \in [0, T]} \norm{\gradpsi(\xbartau{t})} \le C \coloneqq C(\lambda, \alpha, \beta, m_0, T, \tau_0, y_0, x_0).
    \]
    Therefore, by the continuity of $g$ as in \cref{ass:bregman_lower_inequality}, there exists $M > 0$ depending only on the same constants as $C$ such that
    \[
        \inf_{n \le n_T^\tau} g(\norm{\gradpsi(\xtaun)} + \norm{\gradpsi(\xtaunminusone)}) \ge M.
    \]

    From (ii) of \cref{ass:bregman_lower_inequality}, we have
    \[
        \begin{aligned}
            \int_{0}^{T} |D^\tau_r|^2 \dd r 
            &\le \int_0^{n_T^\tau \tau} |D_r^\tau|^2 \dd r
            \le \sum_{n = 1}^{n_T^\tau} g^{-1}(\norm{\gradpsi(\xtaun)} + \norm{\gradpsi(\xtaunminusone)}) \frac{\bregmanast{\gradpsi(\xtaunminusone)}{\gradpsi(\xtaun)}}{\tau}\\
            & \le M \sum_{n = 1}^{n_T^\tau} \frac{\bregman{\xtaun}{\xtaunminusone}}{\tau}
            \le M \sum_{n = 1}^{n_T^\tau} \Phi(\xtaunminusone) - \Phi(\xtaun)
            \le M (\Phi(x_0) - m_0),
        \end{aligned}
    \]
    where we use \eqref{eq:bregman_duality} in the third inequality.
    This completes the proof.
\end{proof}

\subsection{Proof of existence}\label{sec:proof_existence}
To show the existence of $\Psi$-mirror flows, we use a refined version of Ascoli-Arzel\'a's theorem: see~\cite[Proposition~3.3.1]{AGS}.

\begin{lemma}[Ascoli-Arzel\'a's theorem]\label{lem:refined_ascoli_arzela}
    Let $T > 0$, $K \subset V$ be a compact set and $\{x^n_t\}_{n \ge 1}$ be a family of maps $x^n_t \colon [0, T] \to V$.
    If there exists  a symmetric continuous function $A \colon [0, T] \times [0, T] \to [0, +\infty)$ such that $x_t^n \in K$ for each $t \in [0,T]$ and $n \ge 1$, and 
    \[
            \limsup_{n \to + \infty} \norm{x_t^n - x_s^n} \le A(t, s) \text{ for each } 0 \le s \le t \le T,
    \]
    then there exist a continuous map $x_t \colon [0, T] \to V$ and a subsequence $\{x_t^{n(k)}\}_{k\ge1}$ such that $x_t^{n(k)} \to x_t$ for each $t \in [0, T]$.
\end{lemma}

\begin{theorem}\label{thm:existence_MF}
    Assume that \cref{ass:bregman_lower_inequality,ass:psi_norm_bounds,ass:bounded_from_below,ass:bounded_sublevel_compact} hold.
    Let $x_0 \in \domain{\Phi}$.
    Then there exists a $\Psi$-mirror flow $x_t \colon [0, +\infty) \to V$ for $\Phi$ starting from $x_0$.
\end{theorem}
\begin{proof}
    We fix a sufficiently small $\tau_0 > 0$.
    \cref{lem:apriori_boounds} shows that $\{|D^\tau_t|\}_{0 < \tau \le \tau_0} \subset L^2([0,T])$ is bounded for any $T > 0$.
    Thus, by the diagonal argument, there exist a subsequence $\{|D^{\tau_k}_t|\}_{k\ge1}$ and $D \in L^2_{\mathsf{loc}}([0, +\infty))$ such that $\tau_k \to 0$ and $|D^{\tau_k}_t|$ weakly converges to $D$ in $L^2([0, T])$ for any $T > 0$.

    Let $T > 0$.
    For each $0 \le s \le t \le T$, we have
    \[
    \begin{aligned}
        \norm{\gradpsi(\xbar{\tau_k}{t}) - \gradpsi(\xbar{\tau_k}{s})}
        = \norm{\sum_{n = n^{\tau_k}_s + 1}^{n^{\tau_k}_t} \gradpsi(x^{\tau_k}_n) - \gradpsi(x^{\tau_k}_{n-1})}
        \le \int_{n_s^{\tau_k}\tau_k}^{n_t^{\tau_k}\tau_k} |D^{\tau_k}_r| \dd r.
    \end{aligned}
    \]
    By taking limsup with respect to $k\to+\infty$, we obtain
    \begin{equation}\label{eq:limsup_bound_for_AC_loc}
        \limsup_{k\to+\infty}\norm{\gradpsi(\xbar{\tau_k}{t}) - \gradpsi(\xbar{\tau_k}{s})} \le \int_s^t D_r \dd r \text{ for each } 0\le s\le t\le T.
    \end{equation}
    Let $C > 0$ be as in \cref{lem:apriori_boounds}, $\widetilde{K} \coloneqq \{\Phi \le C\} \cap \overline{B_C(0)}$ and $K \coloneqq \gradpsi(\widetilde{K})$.
    Then it holds that $K$ is compact.
    Thus, by \cref{lem:refined_ascoli_arzela} and the diagonal argument, after passing to a subsequence and relabeling, there exists a continuous map $\varphi_t \colon [0, +\infty) \to V^\ast$ such that $\gradpsi(\xbar{\tau_k}{t}) \to \varphi_t$ for each $t \in [0, +\infty)$.
    Moreover, \eqref{eq:limsup_bound_for_AC_loc} implies that $\varphi_t$ is locally absolutely continuous.
    
    We set $x_t \coloneqq \gradpsi^\ast(\varphi_t)$ for each $t \in [0, +\infty)$.
    By the continuity of $\nabla\Psi^\ast$, for each $t \ge 0$ we have $\lim_{k \to +\infty} x^{\tau_k}_t = x_t$.
    For any $0 \le s \le t $, by inserting $n = n^{\tau_k}_t, m = n^{\tau_k}_s$ in \cref{prop:discrete_EVI_exponential} and taking limsup as $k \to +\infty$, we obtain
    \[
        \e^{\lambda(t - s)} \bregman{x}{x_t} - \bregman{x}{x_s} \le (\Phi(x) - \Phi(x_t)) \int_s^t \e^{\lambda(r - s)}\dd r \text{ for each } x \in \domain{\Phi}.
    \]
    By \cref{thm:mirror_flow_equivalent_evi_solution} we get the conclusion.
\end{proof}

\section{\texorpdfstring{$p$}{p}-Laplacian eigenvalue problem on metric measure spaces}\label{sec:p_laplacian}
Let $(X, \dist)$ be a metric space with a Borel measure $\m$ with $\m(X) < +\infty$.
We collectively denote them by $(X, \dist, \m)$.
Let $p \in (1, +\infty)$, $q$ be the H\"older conjugate of $p$ and $\Psi(\cdot) \coloneqq \frac{1}{p}\lpnorm{\cdot}^p$ for the $L^p$-norm $\lpnorm{\cdot}$ of $\Lp$. 
Denoted by $\sphere$ the unit sphere in $\Lp$ centered at $0$ and by $J_p$ the $p$-duality map on $\Lp$, i.e., $J_p(\cdot) = \partial \Psi$.

In this section, we study the eigenvalue problem for the $p$-Laplace operator on $(X, \dist, \m)$ for the case $p \ge 2$.
The proofs of \cref{thm:min_max_eigenvalue_introduction,thm:approximation_eigenvalue_introduction} are given in \cref{sec:ljusternik_schnirelman,sec:approximation_eigenvalues}, respectively.
These results are established by using mirror flows studied in the previous sections.
While the main theorems concern the case $p \ge 2$, the $p$-Laplacian and the associated eigenvalue problem can be formulated for the whole range $p \in (1, +\infty)$.

\begin{remark}
    Throughout this section, our discussion is restricted to the Neumann eigenvalue problem.
    However, the Dirichlet case can be treated in the same way.
\end{remark}
\subsection{Introduction to \texorpdfstring{$p$}{p}-Laplacian and its eigenvalues on metric measure spaces}
We start by introducing the following functional, whose subgradient is related to the $p$-Laplacian.
The $p$-Cheeger energy $\chp \colon \Lp \to [0, +\infty]$ is defined to be that 
\begin{equation}\label{def:p_cheeger_energy}
    \chp(f) \coloneqq \begin{cases}
        \frac1p \int_X \relaxedslope{f}^p \dd \m & \text{if } \relaxedslope{f} \text{exists }\\
        +\infty & \text{otherwise}
    \end{cases},
\end{equation}
where $\relaxedslope{f}$ is the $p$-minimal relaxed slope.
The definition of $\relaxedslope{f}$ is given in, for instance, \cite[Definition 4.2]{density_lipschitz} (there is merely a notational difference: \cite{density_lipschitz} uses $q$, while we use $p$).

\begin{remark}
    There are several equivalent definitions of $p$-Cheeger energy: see~\cite[Theorem~7.1 and Proposition~5.6]{metricsobolev}.
    In particular, $\relaxedslope{f}$ is the minimal $L^p$-norm element among all functions called \textit{(lip)-relaxed $p$-upper slopes} in \cite[Definition~5.9]{metricsobolev}.
\end{remark}

We summarize some properties of $\chp$.
These facts are well known to experts; we refer interested readers to
\cite[Subsection~4.1]{calculus-heatflow}, and point out that similar arguments carry over to the case $p\in(1,+\infty)$ without assuming completeness
or separability of $(X,\dist)$.
\begin{proposition}\label{prop:properties_of_chp}
    The following hold:
    \begin{enumerate}
        \item the functional $\chp$ is convex and lower semicontinuous on $\Lp$.
        \item for any $f \in \Lp$ and constant function $c$ on $X$, one has
            \[
                \chp(c + f) = \chp(f).
            \]
        \item the functional $\chp$ is $p$-homogeneous, i.e., for any $a \in \R$ and $f \in \domain{\chp}$ one has
            \[
                \chp(af) = |a|^p\chp(f).
            \]
        \item for any $f, g \in \domain{\chp}$ and non-decreasing $C^1$-function $\phi \colon \R \to \R$ with $0 \le \phi' \le 1$, one has
            \[
                \chp(f - \phi(f - g)) + \chp(g + \phi(f - g)) \le \chp(f) + \chp(g).
            \]
    \end{enumerate}
\end{proposition}

\begin{remark}
    It is well known that the domain $\domain{\chp}$ is a subspace of $\Lp$ and the function $f \in \domain{\chp} \mapsto \sobolevnorm{f} \coloneqq (\lpnorm{f}^p + p\chp(f))^{1/p}$ is a norm on $\domain{\chp}$.
    Moreover, $\sobolev \coloneqq \domain{\chp}$ equipped with $\sobolevnorm{}$ is a Banach space.
    This space $\sobolev$ is called a \textit{Sobolev space} on $(X, \dist, \m)$.
\end{remark}

We define the $p$-Laplacian as the subdifferential of $\chp$.
Denote by $\partial \chp(f)$ the subdifferential of $\chp$ at $f$.
\begin{definition}[$p$-Laplacian]
    For $f \in \sobolev$ with $\partial \chp(f) \neq \emptyset$, we define the $p$-Laplacian $\plap f$ by the minimal $L^q$-norm element in $-\partial \chp(f)$, namely, $\plap f \in -\partial \chp(f)$ such that $\lqnorm{\plap f} = \min_{\xi \in \partial \chp(f)}\lqnorm{\xi}$.
\end{definition}
Thanks to the $L^q$-closedness of the subdifferential $\partial \chp(f)$ and the uniform convexity of the $L^q$-norm, $\plap f$ is well-defined.

\begin{example}
    This definition of the $p$-Laplacian is compatible with the Euclidean setting.
    Let $X \coloneqq \Omega \subset \R^n$ be a bounded domain with sufficiently smooth boundary, $\dist$ be the Euclidean distance and $\m$ be the Lebesgue measure.
    In this case $\chp$ coincides with the functional
    \[
        \mathcal{E}_p(f) \coloneqq \begin{cases}
            \frac1p \int_X |\nabla f|^p d \m & \text{if } f \in W^{1, p}(\Omega) \\
            +\infty & \text{otherwise}
        \end{cases}.
    \]
    Further, the $p$-Laplacian $\plap f$ satisfies the Euler-Lagrange equation
    \[
        \int (\plap f) g \dd \m = - \int \inner{|\nabla f|^{p-2} \nabla f}{\nabla g} \dd \m \text{ for each } g \in W^{1, p}(\Omega).
    \]
\end{example}

We end this section by defining eigenvalues of the $p$-Laplacian.
\begin{definition}[Eigenvalues and eigenfunctions]
    A real value $\lambda \in \R$ is an eigenvalue of $\plap$ if there exists $f \in \sobolev \setminus \{0\}$ such that
    \[
        \plap f = - \lambda J_p(f).
    \]
    Moreover, such $f$ is called an eigenfunction of $\plap$ for $\lambda$.
\end{definition}

\begin{remark}
    By the $p$-homogeneity of $\chp$, it clearly holds that
    \[
        \inner{\plap f}{ f} = -p\chp(f).
    \]
    Thus, any eigenvalue $\lambda$ of $\plap$ is non-negative.
\end{remark}

\subsection{Characterization of eigenvalues as critical points of a functional \texorpdfstring{$\Phi$}{phi}}
In this section, we characterize eigenvalues of $\plap$ as critical points of a functional.

For each $M > 0$, we denote by $E^M$ the following sublevel set of $\chp$
\[
    E^M \coloneqq \left\{ f \in L^p(X, \m) \mid \chp(f) \le \frac{M}{p} \right\}
\]
and by $E^{1,M}$ the intersection of $E^M$ and the closed unit ball $\overline{\ball}$ in $\Lp$ centered at $0$ 
\begin{equation}\label{eq:ball_sublevel_set}
    E^{1, M} \coloneqq E^M \cap \overline{\ball}.
\end{equation}

For $M > 0$ and $L \coloneqq M +1$, we define a functional $\Phi_M \colon \Lp \to [0, +\infty]$ by
\begin{equation}\label{eq:Phi}
    \Phi_M(f) \coloneqq \begin{cases}
        \chp(f) - \frac{L}{p}\lpnorm{f}^p & \text{if } f \in E^{1,M}\\
        +\infty & \text{otherwise}
    \end{cases}.
\end{equation}
Note that  any real number $L > M$ suffices for the following argument; hence we fix $L \coloneqq M +1$.

The Fr\'echet subdifferential $\partial_F\Phi_M$ admits the following form.
\begin{proposition}
    Let $f \in \overline{\ball}$ with $\chp(f) < \frac{M}{p}$.
    Then it holds that
    \begin{equation}\label{eq:subdifferential_Phi}
        \partial_F \Phi_M(f) = \begin{cases}
            \partial \chp(f) - L J_p(f) & \text{if } f \in \ball\\
            \partial\chp(f) + \R_{\ge 0}J_p(f) - LJ_p(f) & \text{if } f \in \sphere
        \end{cases},
    \end{equation}
    where $\R_{\ge 0}J_p(f) \coloneqq \{ a J_p(f) \mid a \ge 0\}$.
\end{proposition}
\begin{proof}
    For any $A \subset \Lp$ we denote by $\delta_A$ the indicator function
    \[
        \delta_A(f) \coloneqq \begin{cases}
            0 & \text{if } f \in A \\
            +\infty & \text{ if } f \notin A
        \end{cases}.
    \]
    We have
    \[
        \Phi_M = \chp + \delta_{E^M} + \delta_{\overline{\ball}} - L\Psi.
    \]
    Since the functions $\delta_{E^M}$ and $\delta_{\overline{\ball}}$ are lower semicontinuous and convex, we have    
    \[
        \partial_F\Phi_M(f) = \partial(\chp + \delta_{E^M} + \delta_{\overline{\ball}})(f) - LJ_p(f) \text{ for each } f \in \overline{\ball}.
    \]
    Since $0 \in \domain{\chp + \delta_{E^M}} \cap \domain{\delta_{\overline{\ball}}}$ and $\delta_{\overline{\ball}}$ is continuous at $0$, we have
    \[
        \partial(\chp + \delta_{E^M} + \delta_{\overline{\ball}})(f)
        = \begin{cases}
            \partial(\chp + \delta_{E^M})(f) & \text{if } f \in \ball\\
            \partial(\chp + \delta_{E^M})(f) + \R_{\ge 0}J_p(f) & \text{if } f \in \sphere
        \end{cases}.
    \]
    Further, by the lower semicontinuity and convexity of $\chp$, we have that $\partial (\chp + \delta_{E^M})(f) = \partial \chp(f)$ if $\chp(f) < \frac{M}{p}$.
    This completes the proof.
\end{proof}

We can characterize eigenvalues and eigenfunctions as critical points of $\Phi_M$.
Recall the definition~\eqref{eq:local_slope} of the local slope $\slope{\Phi_M}$.
\begin{proposition}\label{prop:critical_points_are_eigen_vectors}
    Let $\lambda \ge 0, M > \lambda$ and $f \in \sphere$.
    Then $f$ is an eigenfunction of $\plap$ for an eigenvalue $\lambda$ if and only if $\slope{\Phi_{M}}(f) = 0$ and $\chp(f) = \frac{\lambda}{p}$.
\end{proposition}
\begin{proof}
    From \cref{prop:slope_is_minimal_norm} and \eqref{eq:subdifferential_Phi}, it holds that $\slope{\Phi_M}(f) = 0$ if and only if there exist $\xi^\ast \in \partial \chp(f)$ and $a^\ast \in \R_{\ge 0}$ such that 
    \[
        \lqnorm{\xi^\ast + (a^\ast - L)J_p(f)} = 0.
    \]
    If $f$ satisfies $\plap f = - \lambda J_p(f)$, we can take $- \plap f$ as $\xi^\ast$ and take $L - p \chp(f) \ge 0$ as $a^\ast$.
    Conversely, if $\slope{\Phi_M}(f) = 0$ and $\chp(f) = \frac{\lambda}{p}$, the above $a^\ast$ must satisfy $a^\ast = L - p\chp(f)$.
    Thus, $\xi^\ast = \lambda J_p(f)$.
    Moreover, since it holds that
    \[
        p\chp(f) = \inner{\xi}{f} \le \lqnorm{\xi} \text{ for any } \xi \in \partial \chp(f),
    \]
    the above $\xi^\ast$ is $-\plap f$.
    This completes the proof.
\end{proof}
\subsection{Properties of mirror flows for \texorpdfstring{$\Phi$}{phi}}
We fix any $M > 0$ and fix $\Phi_M$ in this section; hence we write $\Phi_M$ as $\Phi$ for notational simplicity.
We study the properties of $\Psi$-mirror flows for $\Phi$.
Throughout this section, we always assume the following assumption holds.
\begin{assumption}\label{ass:p_condition_and_rellich_kondrachov}
    The following hold:
    \begin{itemize}
        \item $p \ge 2$.
        \item The subspace $\sobolev$ is compactly embedded in $\Lp$, namely,  
        if $\{f_n\}_{n \ge 1} \subset \Lp$ satisfies
        $\sup_{n \ge 1}\lrbra{\lpnorm{f_n} + \chp(f_n)} < +\infty$, then there exist a subsequence $\{f_{n(k)}\}_{k \ge 1}$ and $f \in \Lp$ such that $\lpnorm{f_{n(k)} - f} \to 0$.
\end{itemize}    
\end{assumption}

\begin{example}
    The compact embedding of $\sobolev$ holds for metric measure spaces which satisfy local doubling and the weak local $(1, p)$-Poincar\'e inequality, and there exist some constants $\alpha > 0$, $A, B \ge 0$ and some point $x_0 \in X$ with
    \begin{equation}\label{eq:uniform_integrability}
        \int_X \dist(x, x_0)^\alpha |f|^p \dd \m \le A \|f\|_{L^p}^p + B \chp(f) \text{ for each } f \in \Lp,
    \end{equation}
    see~\cite{HajlaszKoskela2000SobolevMetPoincare}.
    In particular, \eqref{eq:uniform_integrability} holds whenever $X$ has a finite diameter.
    By known results, examples in this class include $\mathrm{CD}(K,N)$ spaces with $N \in (1, +\infty)$, either when $K > 0$ or when $K\in \R$ and the diameter is finite: see~\cite{Sturm2006GeometryII,Rajala2012InterpolatedMeasures}.
    Other examples, which satisfy the above compact embedding assumption, are $\mathrm{RCD}(K, +\infty)$ spaces with $K > 0$, or with $K \in \R$ and finite diameter: see~\cite{AmbrosioHonda2017NewStability}.
\end{example}

Note that $\Phi$ is $\lambda$-convex for $\Psi$ and $\lambda = -L$ in the sense of \cref{def:lambda_convex} and it is bounded from below by $-L$.
Thus, we can apply \cref{thm:existence_MF} to get the following result.
\begin{proposition}
    For any $f_0 \in \domain{\Phi}$, there exists a $\Psi$-mirror flow $f_t$ for $\Phi$ starting from $f_0$.
\end{proposition}

The following result was proved in \cite{AmbrosioHondaPortegies2018Continuity} for gradient flows in the case $p=2$.
Using mirror flows, we extend their result to the case of $p \ge 2$.
\begin{lemma}\label{lem:stay_in_sphere}
    Let $f_0 \in \sphere$ with $\chp(f_0) < \frac{M}{p}$ and $f_t$ be a $\Psi$-mirror flow for $\Phi$ starting from $f_0$.
    Then $f_t \in \sphere$ for any $t \ge 0$, the map $t \mapsto \chp(f_t)$ is non-increasing and 
    \begin{equation}
        \label{eq:mirror_flow_derivative}
        \dt J_p(f_t) \in -\partial \chp(f_t) + p\chp(f_t) J_p(f_t),\ \Leb\text{-a.e. } t > 0.
    \end{equation}
\end{lemma}
\begin{proof}
    The following argument does not depend on the choice of $\xi \in \partial \chp(f_t)$.
    Thus, by abuse of notation, we simply denote a suitable element $\xi_t \in \partial \chp(f_t)$ by $\partial \chp(f_t)$.
    
    Since the map $t \mapsto \Phi(f_t)$ is non-increasing, we have
    \[
        \chp(f_t) - L\Psi(f_t) \le \chp(f_0) - L < +\infty \text{ for any } t \ge 0.
    \]
    Since $f_t \in \overline{\ball}$, we see that $\chp(f_t) \le \chp(f_0) < \frac{M}{p}$.
    If $f_t \in \sphere$, since $\lqnorm{J_p(f_t)}^q$ takes a maximum value $1$ at $t$, we have $\dt \lqnorm{J_p(f_t)}^q = 0$.
    Note that the $p$-homogeneity of $\chp$ implies that $\inner{\partial \chp(f_t)}{f_t} = p \chp(f_t)$.
    If $f_t \in \ball$, then \eqref{eq:subdifferential_Phi} implies that
    \[
        \dt \lqnorm{J_p(f_t)}^q = L - p\chp(f_t) > 0.
    \]
    Thus, we have that $f_t \in \sphere$ for any $t \ge 0$.
    Combining this with the monotonicity of the map $t \mapsto \Phi(f_t)$ yields that the map $t \mapsto \chp(f_t)$ is non-increasing.

    It is clear from $\lpnorm{f_t} = 1$ for any $t \ge 0$ that $0 = \dt \lqnorm{J_p(f_t)}^q$.
    Thus, by \eqref{eq:subdifferential_Phi} we have
    \[
        0 = -p\chp(f_t) + L - a_t,\ \Leb\text{-a.e. } t > 0.
    \]
    This completes the proof.
\end{proof}
 
In the next theorem, we consider $\Psi$-mirror flows for $\chp - L\Psi$, instead of $\Phi$ ($= \chp + \delta_{E^{1,M}} - L\Psi$).
Before going into the details of the theorem, we briefly show the relationship between $\Psi$-mirror flows for $\chp - L\Psi$ and ones for $\Phi$.
Let $\tilde{f}_t$ be a $\Psi$-mirror flow for $\chp - L\Psi$.
By the $p$-homogeneity of $\chp$ and the $(p-1)$-homogeneity of $\partial \chp$ and $J_p$, direct calculations show that the normalized curve $f_t \coloneqq \tilde{f}_t / \lpnorm{\tilde{f}_t}$ satisfies
\[
    \dt J_p(f_t) \in -\partial \chp(f_t) + p\chp(f_t) J_p(f_t),\ \Leb\text{-a.e. } t > 0,
\]
namely, $f_t$ is a $\Psi$-mirror flow for $\Phi$.
Note that the coefficient of $J_p(\tilde{f}_t)$ is a constant $L$ in the right hand side of the following inclusion:
\[
    \dt J_p(\tilde{f}_t) \in -\partial \chp(\tilde{f}_t) + L J_p(\tilde{f}_t), \ \Leb\text{-a.e. } t > 0.
\]
This motivates treating $\tilde{f}_t$.

\begin{theorem}[$L^1$-Gr\"onwall's inequality]\label{thm:unique_existence_un_normalized}
    Let $f_t, g_t$ be $\Psi$-mirror flows for $\chp - L\Psi$.
    Then, it holds that
    \[
        \norm{J_p(f_t) - J_p(g_t)}_{L^1} \le \e^{L(t-s)} \norm{J_p(f_s) - J_p(g_s)}_{L^1}
        \text{ for each } 0 \le s \le t.
    \]
    In particular, there exists at most one $\Psi$-mirror flow for $\chp - L \Psi$ starting from the same initial point.
\end{theorem}
\begin{proof}
    The choice of $\xi \in \partial \chp(f_t)$ has no effect on the following argument; hence, we simply write $\partial \chp(f_t)$ for a suitable element $\xi_t \in \partial \chp(f_t)$, similarly for $\xi_t \in \partial \chp(g_t)$.
    
    Let $\phi(r) \coloneqq \sgn(r)$ for any $r \in \R$, where we set $\sgn(0) \coloneqq 0$.
    For any $f \in L^p(X, \m)$, the function $x \in X \mapsto \phi(f(x))$ is a subgradient of $\norm{\cdot}_{L^1}$ at $f$.
    Thus, by the upper chain rule for convex functionals, for $\Leb$-a.e.~$t > 0$ we have
    \[
        \begin{aligned}
        \dt \lonenorm{J_p(f_t) - J_p(g_t)} 
        = & \inner{\phi(J_p(f_t) - J_p(g_t))}{-\partial \chp(f_t) + \partial \chp(g_t)} \\
        &+ L \inner{\phi(J_p(f_t) - J_p(g_t))}{J_p(f_t) - J_p(g_t)}.
        \end{aligned}
    \]
    It is enough to show that $\inner{\phi(J_p(f_t) - J_p(g_t))}{-\partial \chp(f_t) + \partial \chp(g_t)} \le 0$.
    This inequality is equivalent to $\inner{\phi(f_t - g_t)}{-\partial \chp(f_t) + \partial \chp(g_t)} \le 0$ because the function $r \in \R \mapsto |r|^{p-2}r$ is strictly increasing.
    By the standard mollification, with no loss of generality, we can assume that $\phi$ is $C^1$, non-decreasing and $0 \le \phi' \le C$ for some constant $C>0$.
    By \cref{prop:properties_of_chp}, we have 
    \[
        \inner{\frac1C\phi(f_t - g_t)}{-\partial \chp(f_t) + \partial \chp(g_t)} \le 0.
    \]
    This completes the proof.
\end{proof}

We go back to the case of $\Phi$.
In order to show the well-posedness of $\Psi$-mirror flows for $\Phi$, we prove the following compactness result for the families of $\Psi$-mirror flows.
For any real number $M'$ with $0 < M' \le M$, we define
\begin{equation}\label{eq:gamma_m_prime}
    \Gamma^{M'} \coloneqq \{ f_t \mid f_t \text{ is a $\Psi$-mirror flow for $\Phi$ with $f_0 \in \sphere \cap E^{M'}$} \}.
\end{equation}

\begin{lemma}[Compactness of $\Gamma^{M'}$]\label{lem:compactness_mirror_flows}
    The collection $\Gamma^{M'}$ is compact in the following sense: for any sequence $\{f_t^n\}_{n \ge 1} \subset \Gamma^{M'}$ there exist a subsequence $\{f_t^{n(k)}\}_{k \ge 1}$ and a $\Psi$-mirror flow $f_t \in \Gamma^{M'}$ such that for any $t_k \in [0, +\infty) \to t \in [0, +\infty)$ we have $\lpnorm{f_{t_k}^{n(k)} - f_t} \to 0$ as $k \to +\infty$.
\end{lemma}

\begin{proof}
    Denoted by $\lesssim_p$ inequalities up to a multiplicative constant depending only on $p$.

    Let us recall \cref{rem:lp_bregman}.
    Note that $ \Phi(f_t^n) \in [-L, M']$ for each $t \ge 0$ and $n \ge 1$.
    \cref{thm:energy_dissipation} shows
    \[
        \int_0^{+\infty} \lqnorm{ \dt J_p(f_t^n)} \dd t \lesssim_p M+L \text{ for each } n \ge 1,
    \]
    namely $\{\lqnorm{\dt J_p(f_t^n)}\}_{n \ge 1} \subset L^2([0, +\infty), \R)$ is bounded.
    Thus, after passing to a subsequence and relabeling, there exists $A \in L^2([0, +\infty))$ such that
    \[
    \limsup_{n \to +\infty} \lqnorm{J_p(f_t^n) - J_p(f_s^n)} \le \int_s^t A(r) \dd r \text{ for any } 0 \le s \le t.
    \]
    This allows us to use \cref{lem:refined_ascoli_arzela} to obtain the relabeled subsequence $\{f_t^n\}_{n\ge1}$ and a curve $f_t$ such that
    $\lpnorm{f_t^n - f_t} \to 0$ for each $t \ge 0$ and the curve $t \in [0, +\infty) \mapsto J_p(f_t)$ is locally absolutely continuous.
    Thanks to the lower semicontinuity of $\chp$, we have $f_0 \in \sphere \cap E^{M'}$.
    Further, since each curve $f_t^n$ satisfies the exponential integral form~\eqref{eq:evi_exponential_integral_form}, so does $f_t$; hence, by \cref{thm:mirror_flow_equivalent_evi_solution}, the curve $f_t$ belongs to $\Gamma^{M'}$.
    Let $t_n \in [0,+\infty)$ be such that $t_n \to t \in [0, +\infty)$.
    In the following inequality induced by the triangle inequality
    \[
        \lpnorm{f_t - f^n_{t_n}}^p
        \lesssim_p \lpnorm{f_t - f_t^n}^p + \lpnorm{f_t^n - f_{t_n}^n}^p
    \]
    the first term converges to $0$ as $n \to +\infty$, and so does the second term by the exponential integral form \eqref{eq:evi_exponential_integral_form}.
    This completes the proof.
\end{proof}

\begin{corollary}[Well-posedness of $\Psi$-mirror flows for $\Phi$]\label{cor:unique_existence_mirror_flow}
    For any $f_0 \in \sphere$ with $\chp(f_0) < \frac{M}{p}$, there exists a unique $\Psi$-mirror flow $f_t$ for $\Phi$ starting from $f_0$, which is denoted by $F_t[f_0]$.
    Moreover, if $\{f_0^n\}_{n \ge 1} \subset \sphere$ with $\sup_{n\ge1}\chp(f_0^n) < \frac{M}{p}$ and $f_0 \in \sphere$ satisfy $\lpnorm{f_0^n - f_0} \to 0$, then  we have
    \[
        F_t[f_0^n] \to F_t[f_0] \text{ in } \Lp \text{ for each } t \ge 0.
    \]
\end{corollary}
\begin{proof}
    We define the function
    $
    a_t \coloneqq \int_0^t (L - p\chp(f_r) )  \dd r
    $
    and the curve $\tilde{f}_t \coloneqq a_t^{\frac{q}{p}}f_t$.
    By \eqref{eq:mirror_flow_derivative} and the $(p-1)$-homogeneity of $\partial\chp$ and $J_p$, a simple calculation of the derivative $\dt J_p(\tilde{f}_t)$ shows that $\tilde{f}_t$ is a $\Psi$-mirror flow for $\chp - L\Psi$.
    Further, we have $f_t = \tilde{f}_t / \lpnorm{\tilde{f}_t}$ because of $\lpnorm{f_t} = 1$.
    Thus, \cref{thm:unique_existence_un_normalized} shows the unique existence of $\Psi$-mirror flow starting from $f_0$.
    The convergence of $F_t[f_0^n]$ to $F_t[f_0]$ is a direct consequence of the unique existence of $F_t[f_0]$ and \cref{lem:compactness_mirror_flows} for $\frac{M'}{p} \coloneqq \sup_{n\ge1} \chp(f_0^n)$.
\end{proof}

Since the maps $f \mapsto \partial \chp(f)$ and $f \mapsto J_p(f)$ are odd, it immediately holds that $- F_t[f_0]$ is a $\Psi$-mirror flow starting from $-f_0$, which is unique by \cref{cor:unique_existence_mirror_flow}.
This proves the following result.
\begin{corollary}
    Let $f_0 \in \sphere$ with $\chp(f_0) < \frac{M}{p}$.
    We have
    \[
        F_t[-f_0] = - F_t[f_0] \text{ for each } t \ge 0
    \]
\end{corollary}

\begin{theorem}[Existence of deformation map]\label{thm:deformation_map}
    Let $A \coloneqq \sphere \cap E^{M'}$ for some $M' \in (0, M)$.
    The map $F \colon [0, 1] \times A \to A$ is defined by the unique $\Psi$-mirror flow $F_t[f_0]$ for $\Phi$ for each $t \in [0, 1]$ and $f_0 \in A$.
    Then we have
    \begin{itemize}
        \item for each $t \in [0, 1]$, the map $F_t \colon A \to A$ is continuous and odd.
        \item for each $f_0 \in A$, it holds that
        \begin{equation}\label{eq:energy_dissipation_in_deformation}
            \Phi(F_1[f_0]) - \Phi(f_0) \le - \int_0^1 C_p\slope{\Phi}^2(F_t[f_0]) \dd t,
        \end{equation}
        where $C_p$ is a positive constant depending only on $p$.
    \end{itemize}
\end{theorem}
\begin{proof}
    \cref{lem:stay_in_sphere} ensures that the range of the map $F$ is contained in $A$.
    The continuity of the map $F_t \colon A \to A$ directly follows from \cref{cor:unique_existence_mirror_flow}.
    Finally, \eqref{eq:energy_dissipation_in_deformation} is a direct consequence of \cref{thm:energy_dissipation} and \cref{rem:lp_bregman}.
\end{proof}
\subsection{Ljusternik--Schnirelman principle for \texorpdfstring{$p$}{p}-Laplacian on metric measure spaces} \label{sec:ljusternik_schnirelman}
We give a proof of \cref{thm:min_max_eigenvalue_introduction}, whose more detailed formulation is given as \cref{thm:min_max_are_eigenvalues}.
The proof provided in this section follows a strategy similar to that in~\cite{AmbrosioHondaPortegies2018Continuity}. 
The key difference is the use of $\Psi$-mirror flows, which allows us to extend their result to the non-Hilbert case $p > 2$.

As in the previous section, we always assume that \cref{ass:p_condition_and_rellich_kondrachov} holds.

We first recall the definition of a topological dimension $\gamma$ for subsets of Banach spaces.

\begin{definition}[Krasnoselskii genus]
    Let $V$ be a Banach space and
    \[
        \F(V) \coloneqq \{ A \subset V \mid A \text{ is closed and symmetric}\}.
    \]
    Then, the function $\gamma \colon \F(V) \to [0, +\infty]$ is defined as follows.
    Let $A \in \F(V)$ be non-empty.
    If there exists $m \in \N$ and an odd continuous map $h \in C^0(A; \R^m)$, then we set
    \[
        \gamma(A) \coloneqq \inf \{m \mid \exists h \in C^0(A; \R^m), h \text{ is odd}\}.
    \]
    Otherwise, we set $\gamma(A) \coloneqq +\infty$.
    Moreover, we set $\gamma(\emptyset) \coloneqq 0$.
    This function $\gamma$ is called the \textit{Krasnoselskii genus}.
\end{definition}

For ease of reference, we recall the properties of the Krasnoselskii genus.
\begin{proposition}[cf. {\cite[Proposition~5.4]{struwe_variational}}]\label{prop:properties_genus}
Let $A, A_1, A_2 \in \F(V)$ and $h \in C^0(A; V)$ be odd and continuous.
Then, the following hold:
\begin{enumerate}
    \item $\gamma(A) \ge 0$; $\gamma(A) = 0$ if and only if $A = \emptyset$.
    \item $A_1 \subset A_2$ implies $\gamma(A_1) \le \gamma(A_2)$.
    \item $\gamma(A_1 \cup A_2) \le \gamma(A_1) + \gamma(A_2)$.
    \item $\gamma(A) \le \gamma \left( \overline{h(A)}\right)$.
    \item If $A$ is compact and $0 \notin A$, then $\gamma(A) < +\infty$ and there exists a symmetric open neighborhood $N$ of $A$ in $V$ such that $\gamma(\overline{N}) = \gamma(A)$.
\end{enumerate}
\end{proposition}

In this study, we consider $\Lp$ as $V$.
For each $k \in \N$, we define the collection of all $k$-dimensional closed symmetric subsets as 
\[
    \F_k \coloneqq \{ A \subset \sphere \mid A \text{ is closed and symmetric with } \gamma(A) \ge k \}.
\]
Then we define the min-max spectrum of $\chp$ as follows
\[
    \frac{\lambda_k}{p} \coloneqq \inf_{A \in \F_k}\sup_{A} \chp.
\]
To show that $\lambda_k$ is an eigenvalue of the $p$-Laplacian if it is finite, we prepare some results.

Let $\lambda > 0$ and $M \coloneqq \lambda + 1$.
We simply write $\Phi_M$ as $\Phi$.
By the compact embedding assumption of $\sobolev$, we can show the following lemma, which plays the role of the Palais--Smale condition in our setting.
\begin{lemma}\label{lem:palais-smale}
    For each $M > 0$ and $T \ge 0$, the set
    \[
        S^{M, T} \coloneqq \left\{ f \in \sphere \mid \chp(f) \le \frac{M}{p}, \slope{\Phi}(f) \le T \right\}
    \]
    is compact in the following sense:
    for any $\{ f_n \}_{n\ge1} \subset S^{M,T}$, there exists a subsequence $\{f_{n(k)}\}_{k \ge 1}$ and $f \in S^{M, T}$ such that
    $\lpnorm{f_{n(k)} - f} \to 0$ and
    \[
        \chp(f_{n(k)}) \to \chp(f).
    \]
\end{lemma}
\begin{proof}
    Since $E^{1, M}$ defined by \eqref{eq:ball_sublevel_set} is compact, after passing to a subsequence and relabeling, there exists $f \in \sphere$ such that $\lpnorm{f_n - f} \to 0$.
    By the lower semicontinuity of $\chp$ and $\slope{\Phi}$ (see \cref{prop:lower_semicontinuity_of_slope}), we have $f \in S^{M, T}$.
    Further, by \cref{prop:lower_semicontinuity_of_slope}, we get
    \[
        \chp(f_n) \le \chp(f) +  T \lpnorm{f_n - f} + L D_{\Psi}(f_n, f),
    \]
    where $D_{\Psi}$ is given by~\eqref{eq:bregman_div}.
    Combining this with the lower semicontinuity of $\chp$ completes the proof.
\end{proof}

As part of the preliminaries, we define several types of sets containing critical points.
We define a set of critical points $K_\lambda$ as follows
\begin{equation}\label{eq:critical_points}
    K_\lambda \coloneqq \left\{ f \in \sphere \mid \chp(f) = \frac{\lambda}{p}, \slope{\Phi}(f) = 0 \right\}.
\end{equation}
Thanks to \cref{lem:palais-smale}, the set $K_\lambda$ is compact.
For each $r >0$, we define a tubular neighborhood $U_{\lambda, r}$ of $K_\lambda$ to be that
\[
    U_{\lambda, r} \coloneqq \{f \in L^p(X, \m) \mid \lpnorm{f - g} < r \text{ for some } g \in K_\lambda\}.
\]
It is clear that for any open neighborhood $N$ of $K_\lambda$ there exists $r > 0$ such that
\begin{equation}\label{eq:open_nbh_inclusions}
    K_\lambda \subset U_{\lambda, r} \subset N.
\end{equation}
For each $\delta \in (0, 1)$ and $a > 0$, we define another set $N_{\lambda, \delta, a}$ as follows
\[
    N_{\lambda, \delta, a} \coloneqq \left\{f \in \sphere \mid \abs{\chp(f) - \frac{\lambda}{p}} \le \frac{\delta}{p}, \slope{\Phi}(f) \le 2\delta a \right\}.
\]
From \cref{lem:palais-smale}, $N_{\lambda, \delta, a}$ is closed and for any open neighborhood $U$ of $K_\lambda$ and $a > 0$ there exists $\delta > 0$ such that
\begin{equation}\label{eq:closed_set_inclusions}
    K_\lambda \subset N_{\lambda, \delta, a} \subset U.
\end{equation}

Finally, for any open sets $U, N \subset L^p(X, \m)$ with $\overline{U} \subset N$, we define the collection of mirror flows passing thorough $N \setminus \overline{U}$ as follows
\[
    \Gamma(U, N) \coloneqq \{ f_t \in \Gamma^M \mid f_t \in \overline{U} \text{ for some } t \in [0, 1) \text{ and }  f_1 \in N^c\},
\]
where $\Gamma^M$ is defined by \eqref{eq:gamma_m_prime}.
For each $f \in \Gamma(U, N)$, we set
\[
    \tconsume{\overline{U}}(f) \coloneqq \max \{t \in [0, 1] \mid f_t \in \overline{U}\} < 1,
\]
and
\[
    \tconsume{N^c}(f) \coloneqq \min \{t \in [\tconsume{\overline{U}}(f), 1] \mid f_t \in N^c\}.
\]
It is clear that $\tconsume{\overline{U}}(f)$ and $\tconsume{N^c}(f)$ exist, that $f_t \in N \setminus \overline{U}$ for each $\tconsume{\overline{U}}(f) < t < \tconsume{N^c}(f)$, and that
\[
    \tconsume{\overline{U}}(f) < \tconsume{N^c}(f),
\]
holds.
We define the time 
\[
    \tconsume{}(f) \coloneqq \tconsume{N^c}(f) - \tconsume{\overline{U}}(f) > 0.
\]
Roughly speaking, $f_t$ takes at least $\tconsume{}(f)$ time to move from $\partial \overline{U}$ to $\partial N^c$.

\begin{lemma}
    We have
    \[
        \tmin(U, N) \coloneqq \inf_{f \in \Gamma(U, N)} \tconsume{}(f)  > 0,
    \]
    where we set $\tmin(U, N) \coloneqq +\infty$ if $\Gamma(U, N)$ is empty.
\end{lemma}
\begin{proof}
    Without loss of generality, we can assume that $\tmin(U, N) < +\infty$.
    We take a minimizing sequence $f^n \in \Gamma(U, N)$ such that $\lim_{n \to + \infty} \tconsume{}(f^n) = \tmin(U, N)$.
    Thanks to \cref{lem:compactness_mirror_flows} and the compactness of the interval $[0, 1]$, after passing to a subsequence and relabeling, there exist $\tconsume{\overline{U}}, \tconsume{N^c} \in [0, 1]$ and a $\Psi$-mirror flow $f \in \Gamma^M$ such that $\tconsume{\overline{U}}(f^n) \to \tconsume{\overline{U}}$, $\tconsume{N^c}(f^n) \to \tconsume{N^c}$,  $f_{\tconsume{\overline{U}}} \in \overline{U}$ and $f_{\tconsume{N^c}} \in N^c$.
    Since $\overline{U} \cap N^c = \emptyset$, we have $\tmin(U, N) = \tconsume{N^c} - \tconsume{\overline{U}} > 0$.
\end{proof}

\begin{theorem}\label{thm:min_max_are_eigenvalues}
    Let $k ,\ell \in \N$.
    If
    \[
        \lambda = \lambda_k = \lambda_{k+1} = \cdots = \lambda_{k + \ell - 1}
    \]
    is finite, then it holds that
    \[
        \gamma(K_{\lambda}) \ge \ell.
    \]
    In particular, if $\ell > 1$, there exist infinitely many eigenfunctions for $\lambda$.
\end{theorem}
\begin{proof}
    By \cref{prop:properties_genus}, there exists an open neighborhood $N$ of $K_\lambda$ with $\gamma(\overline{N}) = \gamma(K_\lambda)$.
    Further, by \eqref{eq:open_nbh_inclusions} we have $K_\lambda \subset \overline{U_{\lambda, r}} \subset N$ for some $r > 0$.
    For $\tmin(U_{\lambda, r}, N)$ and $C_p$ as in \cref{thm:deformation_map}, we set 
    \[
        a \coloneqq \max\{\tmin(U_{\lambda, r}, N)^{-1}C_p^{-1}, C_p^{-1}\} > 0,
    \]
    where we set $(+\infty)^{-1} \coloneqq 0$.
    Then, applying \eqref{eq:closed_set_inclusions} to $U_{\lambda, r}$ and $a$, there exists $\delta > 0$ such that
    \[
        K_\lambda \subset N_{\lambda, \delta, a} \subset U_{\lambda, r} \subset N.
    \]

    Choose a set $A \in \F_{k + \ell - 1}$ with
    \[
        \sup_{A} \chp \le \frac{\lambda + \delta}{p}.
    \]
    We now show that the deformation map $F$ in \cref{thm:deformation_map} satisfies
    \begin{equation}\label{eq:deformation_decrease_energy}
        F_1(A) \subset E^{1, \lambda - \delta} \cup N.
    \end{equation}
    Arguing by contradiction, suppose that there exists $f_0 \in A$ satisfying $\chp(F_1[f_0]) > \frac{\lambda - \delta}{p}$ with $F_1[f_0] \notin N$.
    Note that, by monotonicity we see $\abs{\chp(F_t[f_0]) - \frac{\lambda}{p}} \le \frac{\delta}{p}$ for each $t \in [0, 1]$, and then $F_t[f_0] \notin N_{\lambda, \delta, a}$ if and only if $\slope{\Phi}^2(F_t[f_0]) \ge 2a\delta$.
    If we have $F_t[f_0] \notin N_{\lambda, \delta, a}$ for all $t \in [0, 1]$, the energy dissipation inequality \eqref{eq:energy_dissipation_in_deformation} gives
    \[
        \chp(F_1[f_0]) \le \chp(f_0) - 2C_p a\delta \le \frac{\lambda - \delta}{p}.
    \]
    If we see $F_t[f_0] \in N_{\lambda, \delta, a}$ for some $t \in [0, 1]$, by $F_t[f_0] \in \Gamma(U_{\lambda, r} N)$ we have
    \[
        \chp(F_1[f_0]) - \chp(f_0)
        \le  - \int^{\tconsume{N^c}(F_t[f_0])}_{\tconsume{\overline{U}_{\lambda, r}}(F_t[f_0])} C_p \slope{\Phi}^2(F_t[f_0]) \dd t
        \le - 2 C_p a \delta \tconsume{}(F_t[f_0])
        \le - 2\delta.
    \]
    This proves the claim \eqref{eq:deformation_decrease_energy}.

    It directly follows from the definition of $\lambda_k$ that
    \[
        \gamma(E^{1, \lambda - \delta}) \le k - 1.
    \]
    Thus, (iii) and (iv) of \cref{prop:properties_genus} imply
    \[
        \gamma(\overline{N})
        \ge \gamma(E^{1, \lambda - \delta} \cup \overline{N}) - \gamma(E^{1, \lambda - \delta})
        \ge \gamma\left(\overline{F_1(A)}\right) - (k - 1)
        \ge \gamma(A) - k + 1
        \ge \ell.
    \]
    This shows that $K_\lambda \neq \emptyset$, and $K_\lambda$ is infinite if $\ell > 1$ by the very definition of $\gamma$.
\end{proof}
\subsection{Approximation of eigenvalues}\label{sec:approximation_eigenvalues}
In this section, we give the proof of \cref{thm:approximation_eigenvalue_introduction}, which is divided into several steps.

We always assume that \cref{ass:p_condition_and_rellich_kondrachov} holds
and that 
\[
    \sphere_0 \coloneqq
        \sphere \cap \domain{\chp} \cap \left\{f \in \Lp \mid \int_X |f|^{p-2}f \dd \m = 0\right\} \neq \emptyset.
\]
We define
\[
    \widetilde{\lambda}_{1,p} \coloneqq \inf_{f \in \sphere_0} p\chp(f).
\]
\begin{remark}
    The Dirichlet case is obtained by replacing $\sphere_0$ with $\sphere \cap \domain{\chp}$ throughout all the arguments in this section.
\end{remark}

\begin{remark}
    We use the notation $\widetilde{\lambda}_{1,p}$ to indicate that, in the cases of interest, this eigenvalue plays the role of the first non-zero eigenvalue.
    In the case of the Neumann boundary condition, the first eigenvalue $\lambda_{1,p}$ is trivial and it is attained by constant functions.
    On the other hand, in the case of Dirichlet boundary condition, the first eigenvalue is non-trivial in many cases.
\end{remark}

We fix $f_0 \in \sphere_0$, $M \coloneqq p(\chp(f_0) + 1)$ and the $\Psi$-mirror flow $f_t$ for $\Phi$ starting from $f_0$, where $\Phi$ stands for $\Phi_M$ given by \eqref{eq:Phi}; we suppress the subscript $M$ for simplicity.
Recall that $f_t$ is unique due to \cref{cor:unique_existence_mirror_flow}.

We define the $\omega$-limit set $\omega(f_t)$ of $f_t$ as follows
\[
    \omega(f_t) \coloneqq \{f_\ast \in \Lp \mid \exists t_n \to +\infty \text{ such that } \lpnorm{f_{t_n} - f_\ast} \to 0\}.
\]

By \cref{ass:p_condition_and_rellich_kondrachov}, we can see that $\omega(f_t) \neq \emptyset$.
Further, by the lower semicontinuity and monotonicity of $t \mapsto \chp(f_t)$ (see \cref{lem:stay_in_sphere}), the following limit
\begin{equation}\label{eq:energy_limit}
    \lambda_\ast \coloneqq \lim_{t \to +\infty} p \chp(f_t)
\end{equation}
exists and is finite.
We define the set $K_{\lambda_\ast}$ by \eqref{eq:critical_points}.
From \cref{prop:critical_points_are_eigen_vectors}, if $K_{\lambda_\ast}$ is nonempty, then $\lambda_\ast$ is an eigenvalue of $\plap$ and any element in $K_{\lambda_\ast}$ is an eigenvector for $\lambda_\ast$.

\begin{proposition}
    It holds that
    \[
        \int_X J_p(f_t) \dd \m = 0 \text{ for each } t \ge 0.
    \]
\end{proposition}
\begin{proof}
    Let $h_t \coloneqq \int_XJ_p(f_t) \dd \m$.
    By \eqref{eq:mirror_flow_derivative} we have
    \[
        \dt h_t = p\chp(f_t)h_t,\ \Leb\text{-a.e. } t > 0
    \]
    Thus we get
    \[
        \dt h_t^2 \le 2p\chp(f_0) h_t^2,\ \Leb\text{-a.e. } t > 0.
    \]
    Gr\"onwall's inequality gives the conclusion.
\end{proof}

Combining the above proposition with \cref{lem:stay_in_sphere} gives the following result.
\begin{corollary}
    We have $\omega(f_t) \subset \sphere_0$.
\end{corollary}

To show that $K_{\lambda_\ast}$ contains $\omega(f_t)$, we prepare the following lemma.
\begin{lemma}\label{lem:energy_convergence}
    Let $M' < M$ and $\{g_t^n\}_{n \ge 1} \subset \Gamma^{M'}$, where $\Gamma^{M'}$ is given by \eqref{eq:gamma_m_prime}.
    Then there exist a (relabeled) subsequence $\{g_t^n\}_{n \ge 1}$ and a $\Psi$-mirror flow $g_t$ for $\Phi$ such that $g_t^n \to g_t$ in $\Lp$ for each $t \ge 0$ and \[
        \chp(g_t^n) \to \chp(g_t),\ \Leb\text{-a.e. } t  > 0.
    \]
\end{lemma}
\begin{proof}
    By \cref{lem:stay_in_sphere,lem:compactness_mirror_flows}, after passing to a subsequence and relabeling, there exists a $\Psi$-mirror flow $g_t$ such that, for each $t \ge 0$, $g_t^n \to g_t$ in $\Lp$ and $g_t \in \sphere_0$.
    Thus, by the definition of $\Phi$, we have $\Phi(g_t) = \chp(g_t) - \frac{L}{p}$ for each $t \ge 0$.
    From \eqref{eq:evi_exponential_integral_form}, for each $0 \le s < t$ and $n \ge 1$ we have
    \[
        \e^{-L(t-s)}\bregman{g_t}{g_t^n} - \bregman{g_t}{g_s^n} \le E_{-L}(t-s)(\chp(g_t) - \chp(g_t^n)).
    \]
    By taking $\limsup$ with respect to $n \to +\infty$, we obtain
    \begin{equation}\label{eq:energy_conv_inequality}
        E_{-L}(t-s) \limsup_{n \to +\infty} \chp(g_t^n) \le E_{-L}(t-s) \chp(g_t) + \bregman{g_t}{g_s}
    \end{equation}
    Since $p\ge2$, there exists a constant $C_p > 0$ such that for each $0 \le s < t$
    \[
        \bregman{g_t}{g_s} = \bregman{J_p(g_s)}{J_p(g_t)}\le C_p \lqnorm{J_p(g_s) - J_p(g_t)}^q,
    \]
    for example, see~\cite[Lemma~10.2.1]{AGS}.
    Since $t \mapsto J_p(g_t)$ is $\Leb$-a.e. differentiable in $[0, +\infty)$, dividing both sides of \eqref{eq:energy_conv_inequality} by $t - s$ and taking the limit as $s \nearrow t$ yields
    \[
        \limsup_{n \to +\infty} \chp(g_t^n) \le \chp(g_t),\ \Leb\text{-a.e. } t > 0.
    \]
    The lower semicontinuity of $\chp$ gives the conclusion.
\end{proof}

Having established all the preliminary results needed to prove \cref{thm:approximation_eigenvalue_introduction}, we define the simplicity and isolation of eigenvalues used in its statement, and then proceed to the proof.

\begin{definition}\label{def:isolated}
    An eigenvalue $\lambda$ is said to be \textit{isolated} if there exists $\delta > 0$ such that there are no eigenvalues in the interval $(\lambda, \lambda+\delta]$.
\end{definition}

\begin{definition}
    An eigenvalue $\lambda$ is said to be \textit{simple} if any two eigenfunctions $f, g \in \sobolev\setminus\{0\}$ for $\lambda$ are linearly dependent.
\end{definition}

\begin{proof}[Proof of \cref{thm:approximation_eigenvalue_introduction}]
    It has already been shown that $\lambda_\ast$ exists and is finite.
    
    First, we show that $\omega(f_t) \subset \sphere_0 \cap K_{\lambda_\ast}$.
    This also implies that $\widetilde{\lambda}_{1,p} \le \lambda_\ast$.
    Let $t_n >0$ and $f_\ast \in \Lp$ satisfy $t_n\to +\infty$ and $f_{t_n} \to f_\ast$.
    It is clear that, for each $n \ge 1$, the curve $g_t^n \coloneqq f_{t_n + t}$ is the $\Psi$-mirror flow for $\Phi$ starting from $f_{t_n}$ with $\sup_{n\ge1}\chp(g_0^n) \le \chp(f_0)$.
    By \cref{lem:energy_convergence}, after passing to a subsequence and relabeling, there exists a $\Psi$-mirror flow $g_t$ such that $g_t^n \to g_t$ for each $t \ge 0$, in particular $g_0 = f_\ast$, and
    \[
        \chp(g_t^n) \to \chp(g_t),\ \Leb\text{-a.e. } t > 0.
    \]
    Thus, from \eqref{eq:energy_limit} we have
    \[
        \chp(g_t) = \frac{\lambda_\ast}{p},\ \Leb\text{-a.e. } t > 0.
    \]
    Since the map $t \mapsto \chp(g_t)$ is non-increasing and lower semicontinuous, the above equality holds at every point $t \ge 0$; in particular, $\chp(g_0) = \chp(f_\ast) = \frac{\lambda_\ast}{p}$.
    Noting that $\Phi(g_t) = \chp(g_t) - \frac{L}{p}$ for each $t \ge 0$, by \cref{thm:energy_dissipation}, we have
    \[
        \slope{\Phi}(g_t) = 0,\ \Leb\text{-a.e. } t > 0.
    \]
    The lower semicontinuity of $t \mapsto \slope{\Phi}(g_t)$ implies that $\slope{\Phi}(f_\ast) =0$.
    This proves $\omega(f_t) \subset K_{\lambda_\ast} \cap \sphere_0$.

    By taking a minimizing sequence $\{f_0^n\}_{n \ge 1} \subset \sphere_0$ with $\widetilde{\lambda}_{1,p} \le p\chp(f_0^n) \le \widetilde{\lambda}_{1,p} + \frac{1}{n}$, the compact embedding assumption \cref{ass:p_condition_and_rellich_kondrachov} and the lower semicontinuity of $\slope{\Phi}$ show that $\widetilde{\lambda}_{1,p}$ is an eigenvalue of $\plap$.
    Finally, if $p\chp(f_0) \le \widetilde{\lambda}_{1,p} + \delta$ for $\delta$ as in \cref{def:isolated}, we have $\lambda_\ast = \widetilde{\lambda}_{1,p}$.
    Further, since $\omega(f_t)$ is connected, the simplicity implies that $\omega(f_t)$ is a singleton.
    This completes the proof.
\end{proof}

\phantomsection
\addcontentsline{toc}{section}{\texorpdfstring{\textbf{Acknowledgements}}{Acknowledgements}}
\noindent\textbf{Acknowledgements}\quad
I would like to thank my supervisor, Professor Shouhei Honda, for bringing the work~\cite{AmbrosioHondaPortegies2018Continuity} to my attention and for valuable comments.
I also acknowledge support from JSPS KAKENHI Grant Number JP25KJ0966.

\printbibliography
\end{document}